\documentclass{amsart}
\usepackage{amsthm, amsmath, amsfonts, amssymb, array, enumerate,mathtools, verbatim,enumerate,enumitem,xcolor}
\usepackage{url}
\usepackage[all]{xy}

\newcommand{\isomto}{\overset{\sim}{\rightarrow}}

\newcommand{\bA}{\mathbb{A}}
\newcommand{\CC}{\mathbb{C}}
\newcommand{\ZZ}{\mathbb{Z}}
\newcommand{\TT}{\mathbb{T}}

\newcommand{\QQ}{\mathbb{Q}}
\newcommand{\FF}{\mathbb{F}}
\newcommand{\cO}{\mathcal{O}}

\newcommand{\fp}{\mathfrak{p}}

\newcommand{\fm}{\mathfrak{m}}

\newcommand{\cM}{\mathcal{M}}
\newcommand{\cK}{\mathcal{K}}

\newcommand{\cA}{\mathcal{A}}
\newcommand{\cZ}{\mathcal{Z}}

\newcommand{\cH}{\mathcal{H}}
\newcommand{\cC}{\mathcal{C}}

\DeclareMathOperator{\Hom}{Hom}

\DeclareMathOperator{\Ind}{Ind}
\DeclareMathOperator{\cInd}{c-Ind}
\DeclareMathOperator{\End}{End}
\DeclareMathOperator{\Rep}{Rep}

\DeclareMathOperator{\Frac}{Frac}

\DeclareMathOperator{\diag}{diag}
\DeclareMathOperator{\Gal}{Gal}

\DeclareMathOperator{\ch}{char}

\newtheorem{thm}{Theorem}[section]
\newtheorem{lemma}[thm]{Lemma}
\newtheorem{prop}[thm]{Proposition}
\newtheorem{cor}[thm]{Corollary}

\newtheorem{conj}[thm]{Conjecture}

\newtheorem*{thm*}{Theorem}

\begin{document}

\title[The universal unramified module for $GL(n)$]{The universal unramified module for $GL(n)$ and the Ihara conjecture}

\author{Gilbert Moss}

\date{\today}
\maketitle
\begin{abstract}
Let $F$ be a finite extension of $\mathbb{Q}_p$. Let $W(k)$ denote the Witt vectors of an algebraically closed field $k$ of characteristic $\ell$ different from $p$, and let $\cZ$ be the spherical Hecke algebra for $GL_n(F)$ over $W(k)$. Given a Hecke character $\lambda:\cZ\to R$, where $R$ is an arbitrary $W(k)$-algebra, we introduce the universal unramified module $\cM_{\lambda,R}$. We show $\cM_{\lambda,R}$ embeds in its Whittaker space and is flat over $R$, resolving a conjecture of Lazarus. It follows that $\cM_{\lambda,k}$ has the same semisimplification as any unramified principal series with Hecke character $\lambda$. 

In the setting of mod-$\ell$ automorphic forms of \cite{cht}, Clozel, Harris, and Taylor formulate a conjectural analogue of Ihara's lemma. It predicts that every irreducible submodule of a certain cyclic module $V$ of mod-$\ell$ automorphic forms is generic. Our result on the Whittaker model of $\cM_{\lambda,k}$ reduces the Ihara conjecture to the statement that $V$ is generic.
\end{abstract}

\section{The universal unramified module}

\subsection{Main results} Let $F$ be a finite extension of $\QQ_p$ with residue field of order $q$, and ring of integers $\cO_F$. Let $G:=GL_n(F)$ and let $K:=GL_n(\cO_F)$. Let $k$ be an algebraically closed field of positive characteristic $\ell\neq p$, and $\ell\neq 2$.

Given a $k[G]$-module $\pi$ (always presumed to be smooth), the spherical Hecke algebra $k[K\backslash G/K]$ acts on the submodule $\pi^K$ of $K$-fixed vectors via double-coset operators. Denote this action by $*$.

Let $k[G/K]$ denote the space of finitely supported functions on the set $G/K$, and let $G$ act by left-translation. Then $k[G/K]$ admits a natural action of $k[K\backslash G/K]$ and, in fact, the map $k[K\backslash G/K]\to \End_{k[G]}(k[G/K])$ is an isomorphism (\cite[Prop 1.16]{lazarus}). 

Let $\lambda:k[K\backslash G/K]\to k$ be a homomorphism. We define the universal unramified module for $\lambda$:
$$\cM_{\lambda,k}:=k[G/K]\otimes_{k[K\backslash G/K],\lambda}k.$$ It is universal in the following sense: for any $k[G]$-module $V$, denote
$$\pi^{K,\lambda}:=\{v\in \pi^K : z*v = \lambda(z)v\text{ for all }z\in k[K\backslash G/K]\},$$ then there is an isomorphism $\Hom_{k[G]}(\cM_{\lambda,k},\pi) \cong \pi^{K,\lambda}$, where the map $1_K\otimes 1\mapsto v$ corresponds to $v\in \pi^{K,\lambda}$. 

The module $\cM_{\lambda,k}$ has been studied in \cite{kato_eigen, serre_letters, lazarus, lazarus_these, bellaiche_these, bel_ot, cht, gk_univ}. In \cite{serre_letters,cht}, its properties were applied in the global setting of mod-$\ell$ automorphic representations. The goal of this article is to examine the structure of $\cM_{\lambda,k}$ and its Whittaker model, and apply our findings toward several outstanding questions in both the local and global settings.

Let $B=TU$ be the standard Borel subgroup, where $T$ is the diagonal torus. There is an isomorphism of rings, due to Satake, 
$$k[K\backslash G/K]\cong k[T/T(\cO_F)]^{W_G},$$ where $W_G$ is the Weyl group of $G$, and the ring structure on $k[K\backslash G/K]$ is given by convolution. Any map $\lambda:k[K\backslash G/K]\to k$ corresponds to a Weyl orbit of unramified characters $\chi:T\to k^{\times}$, and we can ask about the connection between $\cM_{\lambda,k}$ and the associated unramified principal series representations 
$$i_B^G\chi:=\{\phi:G\to k \text{ locally constant}:\phi(tug)=\delta_B^{-1}(t)\chi(t)\phi(g), t\in T, u\in U, g\in G\}.$$ The modulus character $\delta_B:B\to k^{\times}$ is defined as
$$\delta_B(b):= [bKb^{-1}:bKb^{-1}\cap K] / [K: bKb^{-1}\cap K],$$ for any choice of compact open subgroup $K$ with pro-order prime to $\ell$ (\cite[I.2.6]{vig}). The following conjecture was made by Lazarus.
\begin{conj}[\cite{lazarus}, \S 2 Remarque]\label{lazarus:conj}
There is an equality of Jordan--Holder multisets $\text{JH}(\cM_{\lambda,k})=\text{JH}(i_B^G\chi)$.
\end{conj}
Conjecture~\ref{lazarus:conj} was already known in many cases: when $n=2$ it follows from results of Serre in \cite{serre_letters}, which are proved with Bruhat--Tits theory. It is proved in \cite{lazarus} when the characteristic $\ell$ is banal for $G$ (i.e. $\ell=0$ or $\ell\nmid \#GL_n(\FF_q)$) and in \cite[Lemma 5.1.4]{cht} for $\ell>n$ and $q\equiv 1$ mod $\ell$. It appears in \cite{bel_ot} for arbitrary $\ell\neq p$ for $GL_3$, again using Bruhat--Tits theory. 
\begin{thm}\label{lazarus:thm}
Conjecture~\ref{lazarus:conj} is true.
\end{thm}
We prove the conjecture by showing $W(k)[G/K]$ is flat over $W(k)[K\backslash G/K]$, where $W(k)$ is the Witt vectors. This flatness was conjectured by Lazarus \cite[Conjecture 1.0.5]{lazarus_these}, and was previously known for $GL_3$ (\cite[\S1.3]{bel_ot}). We emphasize that Theorem~\ref{lazarus:thm} does not contain an assumption that $\ell$ is banal.

It is natural to ask whether the arrangement of the Jordan--Holder constituents of $\cM_{\lambda,k}$ exhibits a consistent structure as we deform $\lambda$. For example, when $n=2$ it was shown by Serre in \cite{serre_letters} that $\cM_{\lambda,k}$ always has a unique irreducible submodule that is infinite-dimensional.

Fix an additive character $\psi:F\to k^{\times}$ with conductor zero, and let $\psi$ also denote the usual extension of $\psi$ to a nondegenerate character $U\to k^{\times}$. We say a $k[G]$-module $\pi$ is $\emph{generic}$ if $\Hom_{k[G]}(\pi,\Ind_U^G\psi)\neq 0$, where $\Ind_U^G\psi$ is the space $$\{W:G\to k\text{ locally constant}: W(ug)=\psi(u)W(g),\ u\in U,\ g\in G\}.$$ The Shintani formula for spherical Whittaker functions (\cite{shintani}) implies that the map 
\begin{align*}
\text{ev}_1:\Ind_U^G\psi&\to k\\
W &\mapsto W(1)
\end{align*} induces an isomorphism $(\Ind_U^G\psi)^{K,\lambda}\isomto k$ (see Section~\ref{sphericalwhittaker:section}). Let $W_{\lambda}^0\in (\Ind_U^G\psi)^{K,\lambda}$ denote the preimage of $1$. The universal property of $\cM_{\lambda,k}$ gives a canonical morphism $\Lambda:\cM_{\lambda,k} \to \Ind_U^G\psi$ sending $1_K\otimes 1$ to $W_{\lambda}^0$.
\begin{thm}\label{universalmodule:conj:k}
For every $\lambda$, the map $\Lambda: \cM_{\lambda,k} \to \Ind_U^G\psi$ is injective.
\end{thm}

In fact, we prove a stronger result over $W(k)$: see Theorem~\ref{universalmodule:conj} below. 

When $n=1$, $k[G/K]$ can be identified with the universal unramified character $F^{\times}\to k[K\backslash G/K]^{\times}\cong k[X^{\pm 1}]^{\times}$, and the result is immediate. When $n=2$, Theorem~\ref{universalmodule:conj:k} is easily deduced from Serre's description (\cite{serre_letters}) and the fact that irreducible representations of $GL_2(F)$ are generic if and only if they are infinite-dimensional. Theorem~\ref{universalmodule:conj:k} was tentatively conjectured by Clozel, Harris, and Taylor (\cite[end of \S 5.1]{cht}).

Only one irreducible Jordan--Holder constituent of $\cM_{\lambda,k}$ is generic (this follows from Theorem~\ref{lazarus:thm} or Proposition~\ref{derivativefreerankone:cor} below). But Theorem~\ref{universalmodule:conj:k} tells us the unique generic constituent must occur as a submodule, and that it is the only irreducible submodule. Representations with this property are \emph{essentially absolutely irreducible generic} in the terminology of Emerton and Helm in \cite[\S3.2]{eh}, where they were studied in the context of formulating the local Langlands correspondence in families.

\subsection{Further questions}

These questions have been posed in more general settings. The assumption that $\ch(F)=0$ is a relic of the same assumption appearing in the reference \cite{h_bern}, which we cite to prove admissibility of the universal module over its endomorphism ring; it is almost certainly not needed there, but this would need to be checked. The assumption $\ell\neq 2$ is mostly for convenience in the introduction, as it is made in \cite{lazarus} to ensure $k$ contains a square root of $q$ for the Satake isomorphism. In the rest of this paper, all the arguments are over $W(k)$, and remain valid when $\ell=2$ by replacing $W(k)$ with $W(k)[\sqrt{q}]$. 

Beyond $GL_n$, Lazarus has conjectured a criterion for the flatness of $R[G/K]$ over $R[K\backslash G/K]$ for $G$ an unramified reductive $p$-adic group and $K$ hyperspecial, in terms of the Weyl action on unramified characters of a minimal Levi (\cite[Conjecture 1.0.5]{lazarus_these}). The criterion is established for banal $\ell$ (\cite{lazarus_these}), and flatness has been shown for arbitrary reductive groups of rank one (\cite[IV]{bellaiche_these}). One can also consider the compact induction $\cInd_K^G\rho$ for nontrival $\rho$. In \cite{gk_univ}, a flatness criterion was investigated for arbitrary reductive $p$-adic groups, nontrivial $\rho$, and $\ell=p$.  

In future work we will try to extend the methods of the present article to these more general settings. We remark that new insights are required to go beyond $GL_n$, or at least to groups $G(F)$ having non-generic cuspidal irreducible representations. This is because Lemma~\ref{cuspidal=generic}, which links Jacquet functors to Whittaker models in the induction arguments for Theorems~\ref{flatness} and \ref{universalmodule:conj}, would fail. For example, there are cuspidal irreducible representations of $Sp_4(F)$ that do not admit a Whittaker model \cite{howe_ps}. 

\subsection{Consequences in the mod-$\ell$ representation theory of $GL_n(F)$}

We can now characterize those unramified principal series with embeddings into $\Ind_U^G\psi$. For $v\in (i_B^G\chi)^K$, let $p_v:\cM_{\lambda,k}\to i_B^G\chi$ be the map $1_K\otimes 1\mapsto v$, then,
\begin{align*}
\text{$v$ is a generator for $i_B^G\chi$}&\iff \text{$p_v$ is surjective}\\
&\iff \text{$p_v$ is injective (since $\text{JH}(M_{\lambda,k})=\text{JH}(i_B^G\chi)$)}\\
&\iff \text{the socle of $i_B^G\chi$ is irreducible generic.}
\end{align*}
In other words, we have the following corollary.
\begin{cor}\label{sphericyclicessaig}
For any unramified character $\chi:T\to k^{\times}$, the following are equivalent:
\begin{enumerate}
\item $i_B^G\chi$ is cyclic, generated by a spherical vector,
\item the unique generic constituent of $i_B^G\chi$ occurs as a submodule, and $i_B^G\chi$ has no other irreducible submodules,
\item $i_B^G\chi$ is isomorphic to $\cM_{\lambda,k}$ where $\lambda:k[K\backslash G/K]\to k$ is the spherical Hecke character associated to $\chi$.
\end{enumerate}
\end{cor}

When the characteristic of $k$ is zero, every Weyl orbit of unramified characters contains a $\chi$ such that $i_B^G\chi$ is isomorphic to $\cM_{\lambda,k}$ (it is the $\chi$ satisfying the ``does-not-proceed'' condition on segments \cite[4.3.2]{eh}). However, when $k$ has positive characteristic, and $i_B^G\chi$ is reducible, there may not exist a character $\chi$ in the Weyl orbit such that $i_B^G(\chi)$ is isomorphic to $\cM_{\lambda,k}$. For example, in the ``limit'' case $\ell>n$ and $q\equiv 1$ mod $\ell$, Vign\'{e}ras has shown that $i_B^G\chi$ is semisimple (\cite[Appendix B, Thm 1 (7)]{cht}), hence will not exhibit the structure of $\cM_{\lambda,k}$ for any $\chi$, if reducible.

Since no proper quotient of $\cM_{\lambda,k}$ is generic, we deduce another striking corollary.
\begin{cor}\label{propertiesneeded}
Suppose $\pi$ is a smooth $k[G]$-module such that
\begin{enumerate}
\item $\pi$ has a generator $v$ in $\pi^{K,\lambda}$ for some homomorphism $\lambda:k[K\backslash G/K]\to k$, and 
\item $\pi$ is generic.
\end{enumerate}
Then the canonical surjection $1_K\otimes 1\mapsto v: \cM_{\lambda,k} \to \pi$ is an isomorphism (in particular, $\pi$ has finite length).
\end{cor}

\subsection{Application to Ihara's lemma}

In the global setting of mod-$\ell$ automorphic forms of \cite{cht}, Clozel, Harris, and Taylor formulate a conjecture known as ``Ihara's lemma'' (\cite[Conjecture I]{cht}). When $n=2$ it is deduced easily from strong approximation, but is open for $n>2$. Assuming the truth of Ihara's lemma, the authors give a proof of a non-minimal $R=\TT$ theorem. The weaker statement $R^{\text{red}}=\TT$, where $R^{\text{red}}$ is the reduced quotient of $R$, was later obtained unconditionally using Taylor's ``Ihara avoidance'' method (\cite{ihara_avoidance}), and was enough for applications to the Sato--Tate conjecture. However, the full $R=\TT$ theorem would have applications to special values of the adjoint $L$-function, would imply that $R$ is a complete intersection, and it would be useful for generalizing the local-global compatibility results of \cite{em_lg}. Ihara's lemma remains a conspicuous missing piece in our understanding of congruences among algebraic automorphic forms of different levels. 

In Section~\ref{applicationtoihara:section} we apply Corollary~\ref{propertiesneeded} to reduce Ihara's lemma to an easier statement. For the sake of this introduction, we give an informal summary of the punchline, postponing the detailed discussion until Section~\ref{applicationtoihara:section}.
 
In this subsection, let $F_{w_0}$ be the completion at a place $w_0$ of the CM field $F$ appearing in the setting of \cite{cht} (or Section~\ref{globalsetup:section} of this paper). Given a mod-$\ell$ automorphic form $f$ (as in \cite[3.4]{cht}), having level $K=GL_n(\cO_{F,w_0})$ at the place $w_0$, one can form the cyclic $k[GL_n(F_{w_0})]$-submodule $$\langle GL_n(F_{w_0})\cdot f\rangle$$ inside the space of mod-$\ell$ automorphic forms having arbitrary level at $w_0$. If $f$ is an eigenform for a ``non-Eisenstein'' maximal ideal $\fm$ of a certain global Hecke algebra away from $w_0$, the Ihara conjecture predicts that \emph{all irreducible submodules of $\langle GL_n(F_{w_0})\cdot f\rangle$ are generic} (see Conjecture~\ref{weakihara:conj} below for the precise statement).

Corollary~\ref{propertiesneeded} gives two reformulations of the Ihara conjecture.
\begin{cor}The following are equivalent:
\begin{enumerate}
\item $\langle GL_n(F_{w_0})\cdot f\rangle$ has a unique irreducible submodule, which is generic, and has no other generic constituents (i.e. it is ``essentially absolutely irreducible generic'' in the sense of \cite{eh}),
\item all irreducible submodules of $\langle GL_n(F_{w_0})\cdot f\rangle$ are generic (i.e. the Ihara conjecture is true),
\item $\langle GL_n(F_{w_0})\cdot f\rangle$ is generic.
\end{enumerate}
\end{cor}

The implications $(1)\implies (2)\implies (3)$ are immediate; the main point is $(3)\implies (1)$. If $f$ is an eigenform for a non-Eisenstein maximal ideal $\fm$ of the Hecke algebra at split places \emph{away from} $w_0$, it turns out (by looking at the lift of the associated Galois representation) that $f$ must also be an eigenvector for the action of the spherical Hecke algebra \emph{at} $w_0$ (this is shown in \cite{cht}-- see Theorem~\ref{localprops} below). In particular, there is a homomorphism $$\cZ_{v_0}:=k[GL_n(\cO_{F,w_0}) \backslash GL_n(F_{w_0}) / GL_n(\cO_{F,w_0})]\xrightarrow{\lambda} k,$$ depending on $\fm$, such that $z*f = \lambda(z)f$ for $z\in \cZ_{v_0}$. Therefore, the representation $\langle GL_n(F_{w_0})\cdot f\rangle$ satisfies conditions (1) and (2) of Corollary~\ref{propertiesneeded}, and (1) follows.

For the application of Ihara's lemma to the $R=\TT$ theorem in \cite{cht} it suffices to know the truth of Ihara's lemma in the quasi-banal setting: $q\equiv 1$ mod $\ell$ and $\ell>n$, or $\ell$ banal (c.f. \cite[Prop 5.3.5]{cht}). In the quasi-banal setting we give a sufficient condition for the genericity of $\langle GL_n(F_{w_0})\cdot f\rangle$ in terms of the dimension of the span of the images of $f$ under certain Iwahori--Hecke operators at $w_0$ (c.f. Corollary~\ref{quasibanaliwahorirestatement}).

In the literature, there have recently been some results on modified versions of Ihara's lemma beyond $GL_2$. For a very specific set of Satake parameters, Thorne proved it in the banal case using torsion vanishing results on the cohomology of Shimura varieties (\cite{thorne_gln}). A reformulation was given in the banal case by Sorensen (\cite{sorensen_ihara}). Boyer proved it under stronger hypotheses than irreducibility of the associated modular Galois representation (\cite{boyer_ihara}), and it was recently generalized to Shimura curves, under stronger hypotheses, by Shotton and Manning (\cite{man_shot}).

\section*{Acknowledgements}
The author is grateful for many helpful conversations with Ramla Abdellatif, Jean-Fran\c{c}ois Dat, David Helm, Alberto M\`{i}nguez, Stefan Patrikis, Gordan Savin, Vincent S\'{e}cherre, and Claus Sorensen. The author is grateful to David Helm for suggesting the induction strategy that ultimately led to the proof of Theorem~\ref{universalmodule:conj}. Many thanks to Guy Henniart for pointing out a mistake in an earlier version.

\section{Whittaker functions of spherical Hecke eigenvectors}\label{sphericalwhittaker:section}

Let $\cZ:=W(k)[K\backslash G/K]$ be the spherical Hecke algebra over $W(k)$, let $R$ be a ring, and let $\lambda:\cZ\to R$ be a homomorphism. Define $\Ind_U^G\psi_R$ to be the set of locally constant functions $W:G\to R$ satisfying $W(ug)=\psi(u)W(g)$, $u\in U$, $g\in G$.

Let $\varpi$ be a uniformizer of $F$ and let $T^{(j)}$ denote the element of $\cZ$ given by the $K$-double coset operator $$[\diag(\underbrace{\varpi,\dots,\varpi}_{j\text{ times}},1,\dots,1)].$$

Given any $n$-tuple $\mu\in \ZZ^n$, we let $\varpi^{\mu}$ denote the matrix $\diag(\varpi^{\mu_1},\dots,\varpi^{\mu_n})$. If $W:G\to R$ is an element of $(\Ind_U^G\psi_R)^K$, it follows from the Iwasawa decomposition that $W$ is entirely determined by its values on the set $\{\varpi^{\mu}:\mu\in \ZZ^n\}$.

Given a partition $\mu$ of length $n$, we define the Schur polynomial $$S_{\mu}(X_1,\dots,X_n):= \frac{|(X_j^{\mu_i+n-i})_{i,j}|}{\prod_{i<j}(X_i-X_j)}.$$ It is a symmetric function in the variables $X_1,\dots, X_n$. If we let $T_1,\dots, T_n$ denote the elementary symmetric functions in the variables $X_1,\dots, X_n$, then $T_1,\dots, T_n$ generate the ring of symmetric functions, and thus we may write $S_{\mu}$ as a polynomial in $T_1,\dots, T_n$ (this dictionary is given explicitly by the Jacobi--Trudi identities in combinatorics). We will let $S_{\mu}(T_1,\dots, T_n)$ denote the Schur polynomial $S_{\mu}$ expressed as a polynomial in the $T_i$'s.

The Satake isomorphism (\cite[Prop 1.6]{lazarus}) gives an isomorphism of rings:
$$\cZ:=W(k)[K\backslash G/K] \cong W(k)[T/T(\cO_F)]^{W_G}\cong W(k)[X_1^{\pm 1},\dots, X_n^{\pm 1}]^{S_n},$$ where $q^{i(i-1)/2}T^{(i)}$ is sent to the $i$'th elementary symmetric function in the $X_i$'s.

The following proposition is a generalization of the main result of \cite{shintani}, and the proof is nearly identical.

\begin{prop}\label{Shintani}
Let $\lambda:\cZ\to R$ be a homomorphism, and let $W$ be an element of $(\Ind_U^G\psi_R)^K$. Suppose that for each $T^{(j)}\in \cZ$,
$$T^{(j)}*W = \lambda(T^{(j)})W.$$ Then  
$$W(\varpi^{\mu}) = q^{\sum_{i=1}^n(i-n)\mu_i}S_{\mu}\big(\lambda(T^{(1)}),\dots,q^{i(i-1)/2}\lambda(T^{(2)}),\dots,q^{n(n-1)/2}\lambda(T^{(n)})\big)\cdot W(1).$$
\end{prop}
\begin{proof}
We will abbreviate $W(\mu):=W(\varpi^{\mu})$. Set $\widetilde{W}(\mu) = q^{\sum_{i=1}^n(n-i)\mu_i}W(\mu)$. Let $$I(j):=\{\epsilon\in \ZZ^n: \epsilon_i\in \{0,1\}\text{ and } \sum_i\epsilon_i = j\}.$$ As a function $\ZZ^n\to R$, we claim that $\widetilde{W}$ satisfies the following conditions:
\begin{enumerate}
\item $\widetilde{W}\big((0,\dots, 0)\big) = W\big((0,\dots,0)\big)$
\item $\widetilde{W}(\mu) = 0$ if $\mu$ is non-dominant,
\item $q^{j(j-1)/2}\lambda(T^{(j)})\widetilde{W}(\mu) = \sum_{\epsilon\in I(j)}\widetilde{W}(\mu+\epsilon)$ if  $\mu$ is dominant, for $1\leq j\leq n$.
\end{enumerate}
The first condition is obvious, and the second follows from the conductor of $\psi$ being $0$. 

For the third condition, set $N_0:=N\cap K$ and $N_{0,\epsilon}:=N_0\cap \varpi^{\epsilon}K\varpi^{-\epsilon}$. Then by \cite[Sublemma]{shintani}, we have the following decomposition into single cosets:
$$K\varpi^{1^j}K = \bigcup_{\epsilon\in I(j)}\bigcup_{x\in N_0/N_{0,\epsilon}}x\varpi^{\epsilon}K.$$ Since $T^{(j)}*W = \lambda(T^{(j)})W$, the third condition follows after computing the order of $N_0/N_{0,\epsilon}$ (cf. \cite[p.181]{shintani}).

As in \cite[p.182]{shintani} (or by an easy induction argument), a function $\widetilde{W}:\mathbb{Z}^n\to R$ satisfying conditions (1), (2), and (3) is uniquely determined. Since the function $$\mu\mapsto S_{\mu}(\lambda(T^{(1)}),\dots,q^{i(i-1)/2}\lambda(T^{(2)}),\dots,q^{n(n-1)/2}\lambda(T^{(n)}))\cdot W((0,\dots,0))$$ also satisfies (1), (2), and (3), by the elementary properties of Schur polynomials, we have proved the result. 
\end{proof}

\begin{cor}\label{shintanispacefreerankone}
Let $\lambda:\cZ\to R$ be a homomorphism. The map $\text{ev}_1:W\mapsto W(1)$ defines an isomorphism $(\Ind_U^G\psi_R)^{K,\lambda}\isomto R$.
\end{cor}
\begin{proof}
The injectivity is Proposition~\ref{Shintani}. The surjectivity is simply observing that, for any $r$ in $R$, the Whittaker function defined by the equation $$W^0_{\lambda,e}(\varpi^{\mu}) := q^{\sum_{i=1}^n(i-n)\mu_i}S_{\mu}(\lambda(T^{(1)}),q\lambda(T^{(2)}),\dots,q^{n(n-1)/2}\lambda(T^{(n)}))\cdot r$$ is a preimage of $r$ in the map $W\mapsto W(1)$.
\end{proof}

Given a $\cZ$-module structure $\lambda:\cZ\to R$, let $W_{\lambda}^0$ denote the preimage of $1\in R$ in the map $$\text{ev}_1:(\Ind_U^G\psi_R)^{K,\lambda}\to R.$$ More explicitly,
$$W_{\lambda}^0(\varpi^{\mu}) := q^{\sum_{i=1}^n(i-n)\mu_i}S_{\mu}(\lambda(T^{(1)}),q\lambda(T^{(2)}),\dots,q^{n(n-1)/2}\lambda(T^{(n)})).$$

As part of an induction argument below, we will require a version of Corollary~\ref{shintanispacefreerankone} for Levi subgroups. Let $P=MN$ be a proper standard parabolic subgroup of $G$ with Levi $M$ and unipotent radical $N$. Let $K_M:=K\cap M$, let $U_M := U\cap M$ and denote by $\cZ_{M}$ the ring $W(k)[K_M \backslash M /K_M]$. If $V$ is a smooth $W(k)[M]$-module, $\cZ_M$ acts via double-coset operators on the $K_M$-invariants $V^{K_M}$. Given $z\in \cZ_M$, we denote this action by $z*v$, for $v\in V^{K_M}$. There is a natural inclusion $\iota:\cZ\to \cZ_M$, which can be realized via Satake as the inclusion $W(k)[X_1^{\pm 1},\dots,X_n^{\pm 1}]^{S_n} \hookrightarrow W(k)[X_1^{\pm 1},\dots,X_n^{\pm 1}]^{W_M}$, where $W_M$ is the subgroup of the Weyl group $W_G\cong S_n$ corresponding to the Levi $M$. 

For a homomorphism $\tilde{\lambda}:\cZ_M\to R'$, we will consider the space 
$$(\Ind_{U_M}^M\psi_{R'})^{K_M,\tilde{\lambda}}:=\{W\in (\Ind_{U_M}^M\psi_{R'})^{K_M}:z*W = \tilde{\lambda}(z)W\text{ for all }z\in \cZ_M\}.$$

\begin{lemma}\label{shintaniforlevi}
The map $W\mapsto W(1)$ defines an isomorphism 
\begin{equation}\label{shintaniforlevi:eqn}
(\Ind_{U_M}^M\psi_{R'})^{K_M,\tilde{\lambda}} \isomto R'.
\end{equation}
\end{lemma}
\begin{proof}
By the Iwasawa decomposition applied to $M$, any element $W\in (\Ind_{U_M}^M\psi_{R'})^{K_M}$ is determined by its values on weights which are dominant within each Levi component. If $M = GL_{n_1}\times\cdots \times GL_{n_r}$ we can identify $\cZ_M \cong \cZ_{n_1}\otimes\cdots \otimes\cZ_{n_r}$, where $\cZ_{n_i}$ is the spherical Hecke algebra for $GL_{n_i}(F)$. The result then follows from the same argument as in Proposition~\ref{Shintani} and Corollary~\ref{shintanispacefreerankone}, applied to each Levi factor.
\end{proof}

\section{Properties of the universal module}
We will consider the universal unramified module in a more general setting than the introduction. Let $W(k)$ denote the Witt vectors of $k$. Given a commutative ring $R$ and a homomorphism $\lambda:\cZ\to R$, we define the universal unramified module to be
$$\cM_{\lambda,R}:=W(k)[G/K]\otimes_{\cZ,\lambda}R.$$ In this section we establish some basic properties of $\cM_{\lambda,R}$ that will be essential in what follows.

\subsection{Top derivative of the universal module}\label{subsection:derivative}
Given a $W(k)$-algebra $R$, we define the functor $(-)^{(n)}:R[G]\text{-Mod}\to R\text{-Mod}$ to be the $U,\psi$-coinvariants, $$V^{(n)}:=V_{U,\psi}:= V/V(U,\psi),$$ where $V(U,\psi)$ is the sub-$R$-module generated by $\{uv-\psi(u)v:u\in U, v\in V\}$. This is the $n$'th ``derivative'' of the Bernstein--Zelevinsky formalism introduced in \cite{b-zI}. The derivative is exact. By the definition of $V(U,\psi)$ and right-exactness of tensor product, $(-)^{(n)}$ is compatible with extension of scalars in the sense that $(V\otimes_{W(k)} E)^{(n)} \cong V^{(n)}\otimes_{W(k)} E$ for any $W(k)$-module $E$.

Considering $W(k)[G/K]$ as a $\cZ[G]$-module, we can compute its $n$'th derivative. 

\begin{prop}\label{derivativefreerankone:cor}
The $\cZ$-module $(W(k)[G/K])^{(n)}$ is free of rank one. In particular, $(\cM_{\lambda,R})^{(n)}\cong R$, and if $E$ is a $\cZ$-module, then $(W(k)[G/K]\otimes_{\cZ}E)^{(n)}\cong E$.
\end{prop}
\begin{proof}
Let $\lambda=\text{id}:\cZ\to \cZ$ be the identity map. By Corollary~\ref{shintanispacefreerankone}, the space $\cZ\cdot W_{\text{id}}^0=(\Ind_U^G\psi_{\cZ})^{K,\text{id}}$ is a free $\cZ$-module of rank 1. On the other hand, by the universal property of the universal unramified module $M:=W(k)[G/K]$, we have
\begin{align}\label{derivativefreerankone}
(\Ind_U^G\psi_{\cZ})^{K,\text{id}} &\cong \Hom_{\cZ[G]}(M,\Ind_U^G\psi_{\cZ}) \\
&\cong \Hom_{\cZ}(M^{(n)},\cZ)
\end{align}
We would be done if we could show that $M^{(n)}$ is a reflexive module, i.e. $M\cong \Hom_{\cZ}(\Hom_{\cZ}(M^{(n)},\cZ))$.
 
Let $\fp$ be a prime ideal of $\cZ$ and let $\lambda:\cZ\to \kappa$ denote the map to its residue field. Corollary~\ref{shintanispacefreerankone} and the universal property again show $$(\Ind_U^G\psi_{\kappa})^{K,\lambda} \cong \Hom_{\kappa}((M\otimes_{\lambda}\kappa)^{(n)},\kappa),$$ from which it follows that $(M\otimes\kappa)^{(n)} = M^{(n)}\otimes\kappa$ is one-dimensional over $\kappa$. By Nakayama, the localization $(M^{(n)})_{\fp}$ is cyclic. Applying this to the generic point $\eta = \{0\}$ of $\cZ$, we have that $(M^{(n)})_{\eta}\neq 0$, hence the annihilator of each localization $(M^{(n)})_{\fp}$ is zero in $\cZ_{\fp}$, hence $(M^{(n)})_{\fp}$ is free of rank 1, in particular it is reflexive. It follows that $M^{(n)}$ is reflexive.
\end{proof}

\subsection{Admissibility of the universal module}

As $W(k)[G/K]$ is the $K$-invariant subspace of $W(k)[G]$ under right-translation, $\cZ$ acts on $W(k)[G/K]$ on the right by convolution. This action commutes with left-$G$-translation, so there is a morphism $$\cZ\to \End_{W(k)[G]}(W(k)[G/K])^{\text{op}},$$ which is an isomorphism (\cite[Prop 1.16]{lazarus}). We will subsequently omit the ``op,'' as everything is commutative.

\begin{prop}\label{admissibility}
$\cM_{\lambda,R}$ is admissible as an $R[G]$-module for any $R$.
\end{prop}
\begin{proof}
We only need to prove it when $R=\cZ$ since admissibility is preserved by extension of scalars.

This follows from the results of \cite{h_bern}. Let $A$ be the center of the category of smooth $W(k)[G]$-modules. $A$ is a commutative ring, which by definition acts $G$-equivariantly on each object in the category in a way that commutes with all morphisms in the category. It is proven in \cite{h_bern} that each finitely generated $W(k)[G]$-module is admissible as an $A[G]$-module (as explained on p. 4 of \cite{h_bern}, this is an immediate consequence of the construction of faithfully projective objects in each block of the category \cite[Cor 11.18, 11.19]{h_bern}, which are admissible over the centers of their respective blocks \cite[Prop 12.7]{h_bern}). 

Since $W(k)[G/K]$ is a cyclic $W(k)[G]$-module, it is admissible as an $A[G]$-module. Since the map $\cZ\to \End_{W(k)[G]}(W(k)[G/K])$ is an isomorphism, the action of $A$ on $W(k)[G/K]$ factors through a ring homomorphism $A\to \cZ$. If $H$ is any compact open subgroup and $v_1,\dots, v_r$ is a set of generators for $W(k)[G/K]^H$ as an $A$-module, then $v_1,\dots, v_r$ is also a set of generators as a $\cZ$-module.
\end{proof}

\subsection{Jacquet module of the universal module}\label{jacquetmoduleofunivmodule}

Let $P=MN$ be a proper parabolic subgroup of $G$, with Levi component $M$ and unipotent radical $N$. If $R$ is a $W(k)$-algebra, let $r_P^G:R[G]\text{-Mod}\to R[M]\text{-Mod}$ be the un-normalized parabolic restriction functor, and for any $V\in R[G]\text{-Mod}$, we let $p_N: V \to r_P^GV$ denote the canonical quotient map of $R[M]$-modules. Note that $r_P^G$ commutes with arbitrary extension of scalars, for the same reasons as $(-)^{(n)}$ (c.f. Subsection \ref{subsection:derivative}).

\begin{lemma}\label{jacquetmoduleM}
There is a map $\Phi: W(k)[G/K] \to W(k)[M/K_M]$ which is surjective and induces an isomorphism of $\cZ[M]$-modules $$r_P^G(W(k)[G/K]) \cong W(k)[M/K_M].$$ Moreover, $\Phi(1_K)=1_{K_M}$.
\end{lemma}
\begin{proof}
We turn to \cite[\S 10]{bushnell_kutzko}, or \cite[2.3]{kato_eigen} for the normalized version. Let $dn$ denote the Haar measure on $N$ normalized so that $K\cap N$ has measure 1. The map given in \cite[Lemma 10.3]{bushnell_kutzko} by $$(\Phi f)(m) = \delta_P(m)\int_Nf(mn)dn\text{,\ \ \  for $m\in M$}$$ makes sense over the base ring $W(k)$, since $W(k)$ contains a square root of $q$. Its $\cZ$-equivariance is immediate. The proof that it induces an isomorphism $$r_P^G(W(k)[G/K]) \cong W(k)[M/K_M]$$ exactly follows the proof of \cite[Lemma 10.3]{bushnell_kutzko} except it is simpler because we are in the special case where the representation of $K$ under consideration is the trivial character on $W(k)$. The fact that $\Phi(1_K)=1_{K_M}$ follows directly from the explicit description of the map $\Phi$.
\end{proof}

\section{Flatness of the universal module}\label{flatness:section}

In this section we prove, for general linear groups, a conjecture of Lazarus that $\cM_{\lambda,R}$ is flat over $R$ (c.f. \cite{ lazarus, lazarus_these, bellaiche_these,bel_ot}). This section is a good warm-up for Section~\ref{mainthm:section}.

We require the following lemma generalizing the fact that cuspidal representations are generic. We will also use it in Section~\ref{mainthm:section}.

\begin{lemma}\label{cuspidal=generic}
Let $R$ be any $W(k)$-algebra, let $V$ be an admissible $R[G]$-module such that $r_P^GV=0$ for all proper parabolic subgroups $P$. Then either $V=0$ or $V^{(n)}\neq 0$. 
\end{lemma}
\begin{proof}
If $\fm$ is a maximal ideal then $r_P^G(V_{\fm}) = (r_P^GV)_{\fm}$ and $(V^{(n)})_{\fm} = V_{\fm}^{(n)}$, so it suffices to prove the result after assuming $R$ is a local ring. 

If the result holds for all finitely generated submodules of $V$, it also holds for $V$ itself, thus without loss of generality we may replace $V$ by a submodule that is finitely generated over $R[G]$. In particular, $V\otimes\kappa(\fm)$ is admissible and finitely generated, hence of finite length. 

Since $r_P^GV$ is zero, so is $r_P^G(V\otimes_R\kappa(\fm))$ for all proper parabolics $P$. Hence the socle $S$ of $V\otimes\kappa(\fm)$ satisfies $r_P^GS = 0$ for all proper parabolics. Therefore $S$ is either zero, or a finite direct sum of irreducible cuspidal $\kappa(\fm)[G]$-modules. Since cuspidals are generic, we have $S=0$ or $S^{(n)}\subset (V\otimes\kappa(\fm))^{(n)}$ is nonzero. If we are in the case where $S=0$, then $V\otimes\kappa(\fm)$ must also be zero, in which case $V=0$ by Nakayama \cite[2.1.7]{eh}. If we are in the case where $V^{(n)}\otimes_R\kappa(\fm)\cong (V\otimes_R\kappa(\fm))^{(n)}\neq 0$, then $V^{(n)}$ cannot be zero.
\end{proof}

\begin{thm}\label{flatness}
For any $W(k)$-algebra $R$ and any homomorphism $\lambda:\cZ\to R$, the module $\cM_{\lambda,R}$ is flat over $R$.
\end{thm}
\begin{proof}
We proceed by induction. For $n=1$, $W(k)[G/K]=\cZ$ and $\cZ\cong W(k)[X_1^{\pm 1}]$. The module $W(k)[G/K]$ is free of rank one over $\cZ$, hence flat. Since flatness is preserved by extension of scalars, so is $\cM_{\lambda,R}$.

For $n>1$, it suffices to prove that for any injection $\phi: E\hookrightarrow E'$ of finitely generated $\cZ$-modules, the map \begin{align*}
W(k)[G/K]\otimes_{\cZ} E &\to W(k)[G/K]\otimes_{\cZ}E'\\
1_K\otimes e &\mapsto 1_K\otimes \phi(e)
\end{align*} is injective. Let $V$ be the kernel of this map. 

By the compatibility of $(-)^{(n)}$ with change of scalars, we can identify $$(W(k)[G/K]\otimes E)^{(n)}\cong E$$ and similarly for $E'$. The map on derivatives $$(W(k)[G/K]\otimes E)^{(n)} \to (W(k)[G/K]\otimes E')^{(n)}$$ is given by $E\to E'$ which is injective, therefore $V^{(n)}=0$.

Let $P=MN$ be a proper parabolic subgroup. By the compatibility of the Jacquet functor $r_P^G$ with change of scalars we can identify $$r_P^G(W(k)[G/K]\otimes_{\cZ} E) \cong r_P^G(W(k)[G/K])\otimes_{\cZ} E,$$ and similarly for $E'$. By Lemma~\ref{jacquetmoduleM}, we may further identify
$$r_P^G(W(k)[G/K]\otimes_{\cZ} E)\cong W(k)[M/K_M]\otimes_{\cZ}E.$$  As $P$ is a proper parabolic subgroup, $M$ is a product of groups $G_1\times\dots\times G_r$ along the diagonal, with $G_i:=GL_{n_i}(F)$ for $n_i<n$. We have a decomposition
$$W(k)[M/K_M]= W(k)[G_1/K_1]\otimes\cdots \otimes W(k)[G_r/K_r]$$ compatible with the decomposition $\cZ_M = \cZ_{G_1}\otimes\cdots \otimes \cZ_{G_r}$, and by applying the induction hypothesis we conclude that $W(k)[M/K_M]$ is flat over $\cZ_M$. Since, in addition, $\cZ_M$ is flat over $\cZ$ we conclude that $W(k)[M/K_M]$ is flat over $\cZ$. Thus the map $$W(k)[M/K_M]\otimes_{\cZ}E\to W(k)[M/K_M]\otimes_{\cZ}E'$$ is injective. Therefore $r_P^G(V)=0$ for all $P$, and the Theorem follows by Lemma~\ref{cuspidal=generic}.
\end{proof}

\begin{cor}\label{freeness}
For any homomorphism $\lambda:\cZ\to W(k)$, the $W(k)$-module $\cM_{\lambda,W(k)}$ is free.
\end{cor}
\begin{proof}
$\cM_{\lambda,W(k)}$ is the direct sum of the submodules $(\cM_{\lambda,W(k)})_i$ constructed in \cite[\S 2.1]{eh}. By admissibility, each $(\cM_{\lambda,W(k)})_i$ is finitely generated over $W(k)$ (\cite[2.1.5]{eh}). Since direct summands of flat modules are flat, $(\cM_{\lambda,W(k)})_i$ is flat over $W(k)$ by Theorem~\ref{flatness}, hence free.
\end{proof}

We now deduce Theorem~\ref{lazarus:thm} of the introduction from Corollary~\ref{freeness}.
\begin{cor}\label{JHofinduced}
For any $\chi$ in the Weyl orbit corresponding to $\lambda:\cZ\to k$, $\cM_{\lambda,k}$ and $i_B^G\chi$ have the same semisimplification.
\end{cor}
\begin{proof}
Choose a lift $\tilde{\chi}:T \to W(k)^{\times}$ and a corresponding lift $\tilde{\lambda}:\cZ\to W(k)$. Let $\overline{K}$ be an algebraic closure of the fraction field of $W(k)$ (it is isomorphic to $\mathbb{C}$ by virtue of its cardinality and characteristic). Without loss of generality, suppose $\chi$ and its lift $\tilde{\chi}$ are chosen in the Weyl orbit in such a way that $i_B^G\tilde{\chi}\otimes_{W(k)}\overline{\cK}$ has all of its subrepresentations generic (this is the does-not-proceed condition of \cite[4.3.2]{eh}). Then $W(k)[G/K]\otimes_{\cZ,\tilde{\lambda}}\overline{\cK}$ embeds in $i_B^G\tilde{\chi}\otimes_{W(k)}\overline{\cK}$ by Theorem~\ref{universalmodule:conj:k} in characteristic 0 (c.f. \cite[Lemma 3.2.2(4)]{eh}). On the other hand, the map $W(k)[G/K]\otimes_{\cZ,\tilde{\lambda}}\overline{\cK}\to i_B^G\tilde{\chi}\otimes_{W(k)}\overline{\cK}$ is also a surjection because for our choice of $\chi$, $i_B^G\tilde{\chi}\otimes_{W(k)}\overline{\cK}$ is also generated by its spherical vector by \cite[Prop 5.1]{lazarus}. (Alternatively, one could apply the result we are currently proving, since it is already known in characteristic zero by \cite{lazarus}). Hence $$W(k)[G/K]\otimes_{\cZ,\tilde{\lambda}}\overline{\cK}\cong i_B^G\tilde{\chi}\otimes_{W(k)}\overline{\cK}.$$ By Corollary~\ref{freeness}, $W(k)[G/K]\otimes_{\cZ,\tilde{\lambda}} W(k)$ is an integral structure. On the other hand, $i_B^G\tilde{\chi}$ is an integral structure since it is admissible and torsion-free over $W(k)$ (hence free). Now apply the Brauer--Nesbitt theorem (\cite{vig_whitt}) to the two $W(k)$-lattices $W(k)[G/K]\otimes_{\cZ,\tilde{\lambda}}W(k)$ and $i_B^G\tilde{\chi}$, and conclude that their mod-$\ell$ reductions have the same semisimplifications.
\end{proof}

\section{Jacquet module of the Whittaker space}\label{jacquetmodule:section}
In this section we will investigate the Jacquet module of the space $\Ind_U^G\psi_R$. First we will recall a result of Bernstein and Zelevinsky from \cite[\S 5]{b-zI}, which holds over a $W(k)$-algebra $R$.
\subsection{Composition of restriction and induction}\label{subsection:composition}
In this subsection, we will extend the classical (twisted) restriction-induction result from \cite[\S 5]{b-zI} to arbitrary $W(k)$-algebras $R$. We refer to \cite{b-zI} for any unexplained notation.

Let $G$ be a locally profinite group posessing a compact open subgroup with pro-order invertible in $W(k)$. Let $\mathcal{C}(G)$ be the category whose objects consist of pairs $(R,V)$, where $R$ is a $W(k)$-algebra and $V$ is a smooth $R[G]$-module, and whose morphisms consist of pairs $(\lambda,\iota)$, where $\lambda:R\to R'$ is a $W(k)$-algebra morphism and $\iota:V\otimes_{R,\lambda}R'\to V'$ is a morphism of $R'[G]$-modules.

Let $B$, $T$, $U$, $\bar{P}$, $M$, $\bar{N}$ be closed subgroups, $\psi$ a $W(k)$-valued character of $U$, and $\eta$ a $W(k)$-valued character of $\bar{N}$. This choice of letters aligns with the rest of the present paper, but not with \cite{b-zI}:
\begin{center}
\begin{tabular}{|c|c|}
\hline
\text{Our Notation} & \text{Notation in \cite{b-zI}}\\
\hline
$B=TU$ & $P=MU$\\
$\bar{P}=M\bar{N}$ & $Q=NV$\\
$\psi:U\to W(k)^{\times}$ & $\theta:U\to \mathbb{C}^{\times}$\\
$\eta:\bar{N}\to W(k)^{\times}$ & $\psi:V\to \mathbb{C}^{\times}$\\
\hline
\end{tabular}
\end{center}
We make the following assumptions: 
\begin{enumerate}
\item $B=TU$, $\bar{P}=M\bar{N}$, $T\cap U = M\cap \bar{N} = \{1\}$, $T$ normalizes $U$ and $\psi$, and $M$ normalizes $\bar{N}$ and $\eta$. Under this assumption, the compatibility of twisted induction and restriction with tensor products gives the normalized functors $i_{U,\psi}: \cC(T) \to \cC(G)$ and $r_{\bar{N},\eta}:\mathcal{C}(G)\to \mathcal{C}(M)$, respectively.
\item $G$ is countable at infinity, $U$ and $\bar{N}$ are limits of compact subgroups
\item the action of $\bar{P}$ on the coset space $X:= B \backslash G$ via $g\cdot Bh = Bhg^{-1}$, has a finite number of orbits. In this case there is an ordering $Z_1,\dots, Z_k$ of the orbits such that $Z_1$, $Z_1\cup Z_2$, $\dots$, $Z_1\cup\dots\cup Z_k$ are open.
\end{enumerate}
For each orbit $Z\subset X$ choose $\bar{w}\in G$ such that $B\bar{w}^{-1}$ is a point in $Z$, and let $w$ denote the inner automorphism $w(g):=\bar{w}g\bar{w}^{-1}$.
\begin{enumerate}[resume]
\item We say a group $H$ is decomposed with respect to $(M,\bar{N})$, or $(T,U)$ respectively if $H\cap M\bar{N} = (H\cap M)(H\cap \bar{N})$, or $H\cap TU = (H\cap T)(H\cap U)$. 

For each orbit $Z$, we suppose that $w(B)$, $w(T)$, $w(U)$ are decomposed with respect to $(M,\bar{N})$, and that $w^{-1}(\bar{P})$, $w^{-1}(M)$, and $w^{-1}(\bar{N})$ are decomposed with respect to $(T,U)$.
\end{enumerate}

Under these assumptions, we define a functor $\Phi_Z:\cC(T)\to \cC(M)$ associated to any orbit $Z$. Consider the condition:
\begin{itemize}
\item[($\bigstar$)] The characters $w(\psi):=\psi(\bar{w}^{-1}(\ \cdot \ ) \bar{w})$ and $\eta$ coincide on $w(U)\cap \bar{N}$
\end{itemize}
If ($\bigstar$) doesn't hold, $\Phi_Z:=0$. If $(\bigstar)$ holds, define $\Phi_Z$ as follows. Set
\begin{align*}
T' &= T\cap w^{-1}(M); & M' &= w(T)\cap M = w(T')\\
U' &= M\cap w(U); & \bar{N}' &= T\cap w^{-1}(\bar{N})\\
\psi' &= w(\psi)|_{U'}; & \eta' &= w^{-1}(\eta)|_{\bar{N}'}.
\end{align*}
Define $\Phi_Z$ by
$$\Phi_Z:= i_{U', \psi'}\circ \varepsilon_2\circ w \circ \varepsilon_1 \circ r_{\bar{N}',\eta'}: \cC(T)\to \cC(M),$$ where $\varepsilon_1$ denotes twisting by $\delta_U^{1/2}\delta_{U\cap w^{-1}(\bar{P})}^{-1/2}$ and $\varepsilon_2$ denotes twisting by $\delta_{\bar{N}}^{1/2}\delta_{\bar{N}\cap w(B)}^{-1/2}$.

\begin{thm}[\S 5.2, \cite{b-zI}]\label{inductionrestriction}
Under the conditions (1)-(4), the functor $$F = r_{\bar{N},\eta}\circ i_{U,\psi}:\cC(T)\to \cC(M)$$ is glued from the functors $\Phi_Z$ where $Z$ runs through the $\bar{P}$-orbits on $X$. More precisely, there is a filtration $0=F_0\subset \dots \subset F_k = F$ such that $F_i/F_{i-1} \cong \Phi_{Z_i}$.
\end{thm}
\begin{proof}
With only one exception, the argument in \cite[\S 5]{b-zI} involves the geometry of the group $G$, and goes through verbatim after replacing $\mathbb{C}$ with an arbitrary $W(k)$-algebra $R$. The exception is the following special case, which is case IV$_2$ in \cite{b-zI}: $U=\{1\}$, $G=\bar{P}$, $M\subset T$, $\rho\in \Rep_R(T)$. The map $p_{\bar{N}'}:\rho \to r_{\bar{N}',\eta}(\rho)$ is the canonical projection, and the morphism
\begin{align*}
\overline{A}: i_{\{1\},1}(\rho) & \to r_{\bar{N}',\eta}(\rho)\\
f &\mapsto \int_{\bar{N}'\backslash \bar{N}}\eta^{-1}(v)p(f(v))dv,
\end{align*}
induces a morphism of $R$-modules
\begin{align*}
A: r_{\bar{N},\eta}(i_{\{1\},1}(\rho)) \to \varepsilon_2r_{\bar{N}',\eta}(\rho) =:\Phi_Z(\rho).
\end{align*}
From the functoriality of $r_{\bar{N}',\eta}$ and its compatibility with extension of scalars, $\overline{A}$ and $A$ in fact define natural transformations of functors.

It remains to check that $A$ is an isomorphism. Following \emph{loc. cit.}, we may reduce to the case $M=\{1\}$, $T=\bar{N}'$, $\eta=1$, and $\rho$ is the regular representation of $\bar{N}'$ on the space of $R$-valued locally constant compactly support functions, $C_c^{\infty}(\bar{N}',R)$. Then by transitivity of induction, $i_{\{1\},1}(\rho)$ is the regular representation of $\bar{N}$ on the space $C_c^{\infty}(\bar{N},R)$. A Haar integral on $\bar{N}$ is a morphism $C_c^{\infty}(\bar{N},R) \to R$ factoring through $r_{\bar{N},1}(C_c^{\infty}(\bar{N},R))$, and likewise for $\bar{N}'$. Thus existence and uniqueness of Haar integrals (\cite[I.2.4]{vig}) implies $r_{\bar{N}',1}(C_c^{\infty}(\bar{N}',R))$ and $r_{\bar{N},1}(C_c^{\infty}(\bar{N},R))$ are free of rank one over $R$. To check $A$ is an isomorphism we need to show it is surjective (not just nonzero, as in \cite{b-zI}). Let $K$ be a compact open subgroup of $\bar{N}$ with pro-order invertible in $W(k)$, let $g$ in $C_c^{\infty}(\bar{N}',R)$ be the characteristic function of $K\cap \bar{N}'$. Then the Haar integral sending $g$ to $1$ defines an isomorphism $r_{\bar{N}',1}(C_c^{\infty}(\bar{N}',R))\cong R$ that sends $p(g)$ to $1$. If $f\in C_c^{\infty}(\bar{N},R)$ is $[K:K\cap \bar{N}']^{-1}$ times the characteristic function of $K\cap \bar{N}'$ then $\overline{A}(f)=\int_{\bar{N}'\backslash \bar{N}}p([K:K\cap \bar{N}']^{-1}\cdot g)$, which is sent to $1$ in $R$. Hence $A$ is surjective.
\end{proof}

\subsection{An argument of Bushnell and Henniart}

We now return to the setting of the rest of the paper. Let $P=MN$ be a standard parabolic subgroup of $G=GL_n(F)$, with Levi component $M$ and unipotent radical $N$. If $R$ is a $W(k)$-algebra, let $r_P^G:\cC(G)\to \cC(M)$ be the un-normalized parabolic restriction functor. If we let $U_M:=U\cap M$, then $\psi_R|_{U_M}$ is again a nontrivial character and we can form the representation $\Ind_{U_M}^M\psi_R$, and the compact induction $\cInd_{U_M}^M\psi_R$, respectively. Given an object $(R,V)$ in $\cC(\{1\})$, it defines an object of $\cC(U)$ via the action of $\psi$, which we denote $\psi_{V}$. In this way we define functors $\Ind_U^G\psi$, $r_P^G(\Ind_U^G\psi)$, and $\Ind_{U_M}^M\psi$, respectively, with source $\cC(\{1\})$ and target $\cC(G)$, $\cC(M)$, and $\cC(M)$, respectively. 

In the notation of Subsection~\ref{subsection:composition}, we let $T=\{1\}$ and $B=U$. We will take $\bar{P}=M\bar{N}$ to be the $M$-opposite parabolic to $P$, and $\eta$ to be the trivial character.

We start by recording a lemma appearing in \cite{bh_whitt}, whose proof works verbatim over an arbitrary $W(k)$-module $R$.
\begin{lemma}[Lemma 2.3, \cite{bh_whitt}]\label{bhlem}
Let $\bar{P} = M\bar{N}$ be the opposite parabolic to $P$ and let $R$ be any $W(k)$-module. If the set $Ug\bar{P}$ supports a function $f:G\to R$ such that $f(uxn) =\psi(u)f(x)$ for $u\in U$, $x\in G$, $n\in \bar{N}$, then $g$ lies in $U\bar{P}$.
\end{lemma}

If an orbit $U\bar{w}\bar{P}$ satisfied condition ($\bigstar$) in Subsection~\ref{subsection:composition}, there would exist a well-defined function $f:U\bar{w}\bar{P}\to R$ satisfying the hypothesis of Lemma~\ref{bhlem}. Hence the only orbit satisfying ($\bigstar$) is the trivial orbit $Z=U\bar{P}$. Although we will not need it, we record the following corollary of Theorem~\ref{inductionrestriction}:
\begin{cor}
The functor $r_{\bar{N},1}\circ i_{U,\psi}:\cC(\{1\}) \to \cC(M)$ is equivalent to the functor $\Phi_{U\bar{P}} = i_{U',\psi'}$ where $U'=U\cap M$ and $\psi'=\psi|_{U'}$. In particular, we also get the un-normalized version $r_{\bar{P}}^G\cInd_U^G\psi \cong \cInd_{U_M}^M\psi$.
\end{cor}
Note that a $\bar{P}$-orbit $U\bar{w}\bar{P}$ in $U\backslash G$ satisfies condition ($\bigstar$) if and only if the $U$-orbit $\bar{P}\bar{w}^{-1}U$ in $\bar{P}\backslash G$ satisfies the analogue of condition ($\bigstar$) after switching the roles of $(B, T, U, \psi)$ with $(\bar{P}, M, \bar{N},\eta)$. Thus, after switching the roles of $(B, \{1\}, U, \psi)$ with $(\bar{P}, M, \bar{N},\textbf{1})$, it follows from Lemma~\ref{bhlem} that the only orbit satisfying ($\bigstar$) is $\bar{P}U$. We get the following corollary by applying Theorem~\ref{inductionrestriction} with the roles of $(B, \{1\}, U, \psi)$ and $(\bar{P}, M, \bar{N},\textbf{1})$ swapped:
\begin{cor}
The functor $r_{U,\psi}\circ i_{\bar{N},1}: \cC(M)\to \cC(\{1\})$ is equivalent to the functor $\Phi_{\bar{P}U} = r_{U',\psi'}$.  In particular, the un-normalized version is $(i_{\bar{P}}^G(-))_{U,\psi}=(-)_{U',\psi'}$.
\end{cor}
Explicitly, given an object $(R,\rho)$ in $\cC(M)$, the isomorphism $r_{U,\psi}\circ i_{\bar{N},1}(\rho)\cong r_{U',\psi'}(\rho)$ is induced by the homomorphism $\overline{A}:i_{\bar{N},1}(\rho)\to r_{U',\psi'}(\rho)$ satisfying the following. If $p_{U',\psi'}: \rho \to r_{U',\psi'}$ denotes the canonical projection, and $f\in i_{\bar{N},1}(\rho)$, then $$\overline{A}(f)=\int_{U\cap \bar{P}\backslash U}\psi^{-1}(u)p_{U',\psi'}(f(u))du.$$

Let $(R',V)$ be any object in $\cC(\{1\})$. Dualizing, we have
\begin{equation}\label{whittfunctional}
\Hom_{\cC(\{1\})}((R,\rho)_{U',\psi'}, (R',V)) \cong \Hom_{\cC(\{1\})}((i_{\bar{P}}^G(R,\rho))_{U,\psi}, (R',V))
\end{equation}
and applying Frobenius reciprocity, we have
\begin{equation}\label{dualiso}
\Hom_{\cC(M)}((R,\rho), \Ind_{U_M}^M\psi_{(R',V)}) \cong \Hom_{\cC(G)}(i_{\bar{P}}^G(R,\rho), \Ind_U^G\psi_{(R',V)}).
\end{equation}

Explicitly, if $\xi\mapsto W_{\xi}$ denotes a morphism $\rho\otimes_RR'\to  \Ind_{U_M}^M\psi_{(R',V)}$ in $\cC(M)$, and $f\mapsto W_f$ denotes the corresponding morphism $i_{\bar{P}}^G(\rho\otimes_RR') \to \Ind_U^G\psi_{(R',V)}$, then for $f\in i_{\bar{P}}^G(\rho\otimes R')$, $W_f$ restricts on $M$ to the function:
\begin{equation}\label{explicit:whitt}
W_f(m) =\int_{U\cap M\backslash U}\psi^{-1}(u)W_{f(u)}(m)du.
\end{equation}

Finally, we apply Bernstein's second adjointness theorem, which holds when $p$ is invertible in $R$, and when $G$ is the $F$-points of a general linear group (or a classical group with $p \neq 2$, or a reductive group of relative rank 1 over $F$) by \cite[Th\'{e}or\`{e}me 1.5]{dat_finitude}. The second adjointness property gives
\begin{equation}\label{secondadjointness}
\Hom_{\cC(G)}(i_{\bar{P}}^G(R,\rho), \Ind_U^G\psi_{(R',V)}) \cong \Hom_{\cC(M)}((R,\rho), r_P^G\Ind_{U}^G\psi_{(R',V)}).
\end{equation}
Combining this with Equation~(\ref{dualiso}), we have, for each $(R',V)$ in $\cC(\{1\})$, an isomorphism of functors
$$\Hom_{\cC(M)}(-,\Ind_{U_M}^M\psi_{(R',V)})\cong \Hom_{\cC(M)}(-, r_P^G\Ind_{U}^G\psi_{(R',V)}).$$ Yoneda's lemma gives a natural isomorphism of functors
$$\Ind_{U_M}^M\psi\cong  r_P^G\Ind_{U}^G\psi.$$ We summarize this result in the following proposition.

\begin{prop}\label{philemma1}
Consider $\Ind_U^G\psi$, $\Ind_{U_M}^M\psi$, and $r_P^G(\Ind_U^G\psi)$ as functors $$\cC(\{1\})\to \cC(M).$$ There exists a surjective natural transformation $\Phi':\Ind_U^G\psi\to \Ind_{U_M}^M\psi$, which induces an isomorphism $$r_P^G(\Ind_U^G\psi) \cong \Ind_{U_M}^M\psi.$$ In particular, for $(R,V)$ and $(R',V')$ in $\cC(\{1\})$ and a morphism $\lambda:R\to R'$, $\iota:V\otimes_RR'\to V'$ between them, the following diagram commutes:
\begin{equation*}
\xymatrix{
\Ind_U^G\psi_{V} \ar[r]^{\Phi'_{(R,V)}} \ar[d]^{\iota_*}& \Ind_{U_M}^M\psi_{V}\ar[d]^{\iota_*}\\
\Ind_U^G\psi_{V'} \ar[r]^{\Phi'_{(R',V')}} & \Ind_{U_M}^M\psi_{V'},
}
\end{equation*}
where $\iota_*$ denotes pushforward of Whittaker functions along $V\to V\otimes_RR' \xrightarrow{\iota} V'$.
\end{prop}
We will only ever use Proposition~\ref{philemma1} in the special case where $V=R$ and $V'=R'$, at which point we will abbreviate the morphism $\Phi_{(R,R)}'$ by simply $\Phi_R'$.

We note that the isomorphism $r_P^G(\Ind_U^G\psi) \cong \Ind_{U_M}^M\psi$ depends on the choice of Haar measure in the map $\overline{A}$. Unfortunately, we are currently unable to write down a more explicit description of the map $\Phi'$. However, its existence, naturality, and the useful properties of the next subsection fill in everything we will need.

\section{Hecke algebras}
\subsection{Iwahori--Hecke algebras} We first summarize some well-known facts about Iwahori--Hecke algebras that we will need. The standard Iwahori subgroup $I$ consists of the matrices in $K$ whose reduction modulo $(\varpi)$ is upper triangular. Define $$\Lambda = T/T(\cO_F) = \left\{\left(\begin{matrix}\varpi^{\mu_1}\\
&\ddots \\ && \varpi^{\mu_n} \end{matrix} \right): \mu_1,\dots,\mu_n\in \ZZ\right\}.$$ The Satake map induces an isomorphism
\begin{align}\label{satake}
S: \cZ & \isomto W(k)[\Lambda]^{W_G},\\
f &\mapsto \left[t\mapsto \delta_B(t)^{-1/2}\int_Uf(tu)du\right]\nonumber
\end{align}
where $W_G$ denotes the Weyl group of $G$, and the algebra structure on the left is convolution of Hecke operators. The Haar measure $du$ is normalized so that the measure of $U\cap K$ is 1.

The module $W(k)[I\backslash G]$ of finitely supported functions on $I\backslash G$ has an action of $G$ by right translation and its $K$-fixed vectors, $W(k)[I\backslash G]^K = W(k)[I\backslash G/K]$ carry a Hecke action of $\cZ$. The quotient map $I\backslash G \to K\backslash G$ induces an inclusion
$$W(k)[K\backslash G] \hookrightarrow W(k)[I\backslash G],$$ which is $(\cZ,*)$-equivariant on $K$-fixed vectors. By \cite[Prop 1.12]{lazarus}, the Satake isomorphism extends to an isomorphism of $(\cZ,*)$-modules:
\begin{equation}\label{satake-kato}
W(k)[I\backslash G/K] \isomto W(k)[\Lambda].
\end{equation}

The Iwahori Hecke algebra $\cH(G,I)=W(k)[I\backslash G/I]$ admits a presentation due to Iwahori--Matsumoto. Let $\{s_1,\dots, s_{n-1}\}$ denote the transposition matrices, and let $t$ denote the matrix
$$t = \left(\begin{matrix}&1\\
&&\ddots\\
&&&1\\
\varpi
\end{matrix}
 \right).$$ The Iwahori double cosets in $I\backslash G/I$ can be uniquely represented as $t^aw$ where $w$ is product of $s_i$'s and their inverses. Let $S_i:=[Is_iI]$ and $T=[ItI]$ in $\cH(G,I)$, where $[-]$ denotes the characteristic function. Then $\cH(G,I)$ is generated as an algebra by $\{S_1,\dots, S_{n-1}\}\cup \{T,T^{-1}\}$, subject to the usual braid relations and quadratic relations, along with the relations $TS_i=S_{i-1}T$ and $T^2S_1 = S_{n-1}T^2$ (c.f. \cite[I.3.14]{vig}). The subalgebra of $\cH(G,I)$ generated by $\{S_1,\dots, S_{n-1}\}$ is precisely $W(k)[I\backslash K/I]$. We denote this subalgebra $\cH_W$. We let $\textbf{1}:\cH_W\to W(k)^{\times}$ denote the ``trivial'' character $w\mapsto [IwI:I]$.

$\cH(G,I)$ admits another presentation due to Bernstein. Let $$t_j:=\left(\begin{matrix}\varpi I_j&\\&I_{n-j} \end{matrix}\right),\ \ 0\leq j \leq n,$$ and let $T_j = [It_jI]\in \cH(G,I)$. Then $T_j$ is invertible in $\cH(G,I)$, we can define $X_j:=q^{j-(n+1)/2}T_jT_{j-1}^{-1}$, and $\cH(G,I)$ is generated by $\{S_1,\dots, S_{n-1}\}\cup \{X_1^{\pm 1},\dots, X_n^{\pm 1}\}$ subject to the relations
\begin{enumerate}
\item $W(k)[X_1^{\pm 1},\dots, X_n^{\pm 1}]$ is commutative,
\item $X_jS_i = S_iX_j$ except if $j=i$ or $i+1$,
\item $S_iX_{i+1}S_i=qX_i$
\item $X_iS_i = S_iX_{i+1} + (q-1)X_i$
\item $X_{i+1}S_i = S_iX_i - (q-1)X_i$
\item the quadratic relations and braid relations on the $S_i$'s.
\end{enumerate}
If $\mu_1\geq \cdots \geq \mu_n$ is a dominant weight, and $\varpi^{\mu}:=\diag(\varpi^{\mu_1},\dots,\varpi^{\mu_n})\in \Lambda$, then $[I\varpi^{\mu}I]$ is invertible in $\cH(G,I)$. If $\mu$ is non-dominant, it can be written $\mu'-\mu''$ for $\mu',\mu''$ dominant, and the product $[I\varpi^{\mu'}I]*[I\varpi^{\mu''}I]^{-1}$ depends only on $\mu$. This defines an injective map of algebras
\begin{align*}
W(k)[\Lambda]&\hookrightarrow \cH(G,I)\\
\varpi^{\mu}&\mapsto q^{-l(\varpi^{\mu})/2}[I\varpi^{\mu'}I]*[I\varpi^{\mu''}I]^{-1},
\end{align*}
where $l(\varpi^{\mu})=l(w)$, after writing $\varpi^{\mu}=t^aw$ for $w\in W_G$, and $l(w)$ denotes the length of $w$. Note that if $\varpi^{\mu_j}:=\diag(1,\dots, 1,\varpi,1,\dots,1)$, with $\varpi$ in the $j$'th position, this map sends $\varpi^{\mu_j}\mapsto X_j$ (despite this map, note that $X_j$ is not $[I\varpi^{\mu_j}I]$). The image of the algebra $W(k)[\Lambda]$ is the subring $\cA:= W(k)[X_1^{\pm1},\dots,X_n^{\pm1}]$ of $\cH(G,I)$, which is a Laurent polynomial ring. It identifies the center of $\cH(G,I)$ with the image of the subring $W(k)[\Lambda]^{W_G}$, i.e. the Laurent polynomials in $\cA$ invariant under permuting variables (\cite[Lemme I.3.15]{vig}).

We identify the ring $\cZ:=W(k)[K\backslash G/K]$ and its module $W(k)[K\backslash G/I]$ with the ring $W(k)[\Lambda]^{W_G}$ and its module $W(k)[\Lambda]$ via the Satake isomorphisms of Equations~(\ref{satake}) and (\ref{satake-kato}). By the Bernstein presentation, we can and do consider them as subalgebras, $\cZ$ and $\cA$, of $\cH(G,I)$.

Let $P=MN$ be a standard parabolic subgroup. It is the set of block-upper triangular matrices associated to an ordered partition $n=n_1+\cdots +n_r$, and $M\cong GL_{n_i}(F)\times\cdots \times GL_{n_r}(F)$. We have, for each factor $G_i:=GL_{n_i}(F)$, the Bernstein presentation of its Iwahori--Hecke algebra $\cH(G_i,I_i)$, with generators that we will denote $X^{(i)}_1,\dots, X^{(i)}_{n_i}$, $S^{(i)}_1,\dots, S^{(i)}_{n_i-1}$. Let $I_M:=I\cap M$. There is a unique algebra homomorphism \begin{align*}
j:=j_P^G:\cH(M,I_M)&\to \cH(G,I)\\
X_j^{(i)}&\mapsto X_{n_1+\cdots + n_{i-1} + j}\\
S_j^{(i)}&\mapsto S_{n_1+\cdots + n_{i-1} +j}
\end{align*}
If $W_M$ denotes the Weyl group of $M$, considered as a subgroup of $W_G$, then it follows (\cite[I.3]{mw_zel}) from the presentation that
$$\cH(G,I) = \bigoplus_{w\in W_M\backslash W_G}j_P^G(\cH(M, I_M))*[IwI].$$ We identify $\cZ_M:=W(k)[K_M\backslash M/K_M]$ with $W(k)[\Lambda]^{W_M}$ via Satake, and with the center of $\cH(M,I_M)$ via the Bernstein presentation. Each $j:=j_P^G$ gives an embedding $\cZ\hookrightarrow \cZ_M\hookrightarrow \cH(G,I)$ which agrees with the Bernstein map $\cZ\hookrightarrow \cH(G,I)$.

%
%
%

If $I_M$ denotes $I\cap M$, the natural map $V\to r_P^GV$ induces, in characteristic zero or banal characteristic, an isomorphism of $\cH(M,I_M)$-modules $V^I\cong (r_P^GV)^{I_M}$, showing that the analogue of $r_P^G$ for Iwahori--Hecke modules is restriction of scalars from $\cH(G,I)$ to $\cH(M,I_M)$. Actually, this is also true if $R$ is a commutative ring with $q\in R^{\times}$:

\begin{lemma}\label{iwahoriisom}
Let $G$ be a general linear group or a classical group (with $p\neq 2$), and suppose $q$ is invertible in $R$. If $I$ is the Iwahori subgroup of $K$, and $I_M:=I\cap M$, then for any $R[G]$-module $V$, the quotient map $p_N:V\to r_P^G(V)$ induces an isomorphism $$V^I \isomto (r_P^GV)^{I_M}.$$
\end{lemma}
\begin{proof}
The injectivity was proved for arbitrary reductive groups in \cite[II.3.1]{vig}. The surjectivity is \cite[Cor 3.9(i)]{dat_finitude}, on general linear groups and (when $p\neq 2$) classical groups.
\end{proof}

\subsection{Additional structure in the quasi-banal setting}\label{quasibanal:section}

In this entire section, we assume that $\ell$ is \emph{quasi-banal}, which means:
\begin{gather*}
\text{$\ell$ is banal for $GL_n(F)$}\\
\text{or}\\
\text{$\ell>n$ and $q\equiv 1$ mod $\ell$.}
\end{gather*}
We will use this subsection later, in a global context, to deduce Corollary~\ref{quasibanaliwahorirestatement}. Let $\cH(G,I)_k=\cH(G,I)\otimes_{W(k)}k=k[I\backslash G/I]$ denote the Iwahori--Hecke algebra over $k$, and similarly let $\cA_k$, $\cH_{W,k}$, and $\cZ_k$ denote the reductions mod-$\ell$ of the subrings $\cA$, $\cH_W$, and $\cZ$.  The relations (1)-(6) in the Bernstein presentation simplify when $q=1$ mod $\ell$. In the quasi-banal setting, Vign\'{e}ras has proved the following mod-$\ell$ analogue of the well-known result over $\CC$:

\begin{prop}[Vign\'{e}ras, \cite{cht} Appendix B]
Let $\ell$ be quasi-banal. The functor $V\mapsto V^I$ defines an equivalence of categories from the abelian subcategory of smooth $k[G]$-modules generated by their $I$-fixed vectors to the category of $\cH(G,I)_k$-modules.
\end{prop}

\begin{prop}\label{Ifixedcyclic}
Let $\ell$ be quasi-banal. There is an isomorphism of $\cH(G,I)_k$-modules $$k[G/K]^I\cong \cH(G,I)_k\otimes_{\cH_{W,k}}\textbf{1}.$$
\end{prop}
\begin{proof}
To ease notation we will drop $k$'s from the subscripts, so for example $\cH(G,I)=\cH(G,I)_k$. The module $k[G/I]$ has $I$-fixed vectors $k[I\backslash G/I]$, which is cyclic as an $\cH(G,I)$-module, generated by $1_I$. There is a surjection $k[G/I]^I \to k[G/K]^I$ given by $$\phi(x)\mapsto \frac{1}{(K:I)}\sum_{g\in K/I}f(gx).$$ The map is well-defined and surjective because $(K:I)$ is invertible in $k$ as a result of the quasi-banal hypothesis. It follows that $k[G/K]^I$ is a cyclic $\cH(G,I)$-module, generated by the characteristic function $1_K$.

Since $1_K$ is fixed by $K$, the action of $\cH_W$ on $1_K$ is via $\textbf{1}$. Thus $1\mapsto 1_K$ defines a morphism of $\cH_W$-modules $\textbf{1}\to k[G/K]^I$. By the adjunction 
$$\Hom_{\cH(G,I)}(\cH(G,I)\otimes_{\cH_W}\textbf{1}, k[G/K]^I)\cong \Hom_{\cH_W}(\textbf{1},k[G/K]^I),$$ we get a morphism of $\cH(G,I)$-modules $\Theta: \cH(G,I)\otimes_{\cH_W}\textbf{1}\to k[G/K]^I$ given by sending $h\otimes 1$ to $h* 1_K$. Since $1_K$ is a cyclic generator of $k[G/K]^I$, the map $\Theta$ is surjective. However, since $\cA\cong k[\Lambda]$ is free of rank $n!$ over $\cZ\cong k[\Lambda]^{W_G}$, and $\cH(G,I) = \cA\otimes \cH_W$ (as $\cZ$-module), we conclude that $\cH(G,I)\otimes_{\cH_W}\textbf{1}$ is free of rank $n!$ over $\cZ$. On the other hand, Equation (\ref{satake-kato}) identifies $k[G/K]^I = k[I\backslash G/K]\cong k[\Lambda]$, showing that $k[G/K]^I$ is also free of rank $n!$ as a $\cZ$-module. It follows that $\Theta$ is an isomorphism.
\end{proof}

\begin{cor}
Let $\cA_k$ be an $\cH(G,I)_k$ module via the projection $$\cH(G,I)_k\twoheadrightarrow \cH(G,I)_k\otimes_{\cH_{W,k}}\textbf{1}\cong \cA_k.$$ Let $\lambda:\cZ_k\to k$ be a homomorphism. Then $\cM_{\lambda,k}^I$ is isomorphic as an $\cH(G,I)_k$-module to the $n!$-dimensional $k$-algebra $\cA_k\otimes_{\cZ_k,\lambda}k$. 
\end{cor}
\begin{proof}
Using the proof of \cite[Lemma 5.1.4]{cht}, we find that $$\cM_{\lambda,k}^I:=(k[G/K]\otimes_{\cZ,\lambda}k)^I = (k[G/K]^I)\otimes_{\cZ,\lambda}k.$$ Since $k[G/K]^I\cong \cH(G,I)_k\otimes_{\cH_{W,k}}\textbf{1}$ is isomorphic to $\cA_k$ as an $\cA_k$-module, the result follows.
\end{proof}

\subsection{Iwahori-fixed Whittaker functions}
Let $W_{\lambda}^0$ be the canonical element of $\Ind_U^G\psi_R$ guaranteed by Lemma~\ref{Shintani}, and let $$\Lambda_R:\cM_{\lambda,R}\to \Ind_U^G\psi_R$$ be the associated Whittaker model $1_K\otimes 1\mapsto W^0_{\lambda}$. The largest technical obstacle to proving our main result, Theorem~\ref{universalmodule:conj} in Section~\ref{mainthm:section}, is showing that $\Lambda_R$ is compatible with parabolic restriction, in the following sense.  Let $P=MN$ be a proper parabolic subgroup. On one hand, there is the map $\Psi$ such that the following diagram of $\cZ[M]$-modules commutes:
\begin{equation}\label{maindiagram}
\xymatrix{
W(k)[G/K] \ar[r]^{\Lambda_{\cZ}}\ar[d]^{\Phi} &\Ind_U^G\psi_{\cZ} \ar[d]^{\Phi'_{\cZ}}\\
W(k)[M/K_M] \ar@{-->}[r]^{\Psi} & \Ind_{U_M}^M\psi_{\cZ},
}
\end{equation}
where the vertical maps are given by Lemmas~\ref{jacquetmoduleM} and Proposition~\ref{philemma1}. Namely $\Psi$ is the composition
$$W(k)[M/K_M] \cong r_P^G(W(k)[G/K]) \xrightarrow{r_P^G\Lambda_{\cZ}} r_P^G(\Ind_U^G\psi_\cZ)\cong \Ind_{U_M}^M\psi_{\cZ},$$ and $\Psi(1_{K_M}) = \Phi'_{\cZ}(W^0_{\text{id}})$, where $\text{id}:\cZ\to \cZ$ is the identity map and $W_{\text{id}}^0$ is the canonical Whittaker function given by Corollary~\ref{shintanispacefreerankone}.

On the other hand, there is the canonical map 
\begin{align*}
\Lambda^M_{\cZ_M}:W(k)[M/K_M]&\to \Ind_{U_M}^M\psi_{\cZ_M}\\
1_{K_M}&\mapsto W_{\text{id}}^{0,M},
\end{align*} where $\text{id}:\cZ_M\to \cZ_M$ is the identity and $W_{\text{id}}^{0,M}$ is the canonical element of $(\Ind_{U_M}^M\psi_{\cZ_M})^{K_M,\text{id}}$ satisfying $W_{\text{id}}^{0,M}(1)=1$, given by Corollary~\ref{shintaniforlevi}. To make the induction argument work in Section~\ref{mainthm:section}, we need to show the two maps agree in the following sense.

\begin{prop}\label{differbyunit}
Let $j^*:\Ind_{U_M}^M\psi_{\cZ}\to \Ind_{U_M}^M\psi_{\cZ_M}$ be pushforward along the canonical inclusion $j:\cZ\hookrightarrow \cZ_M$. The composition $$W(k)[M/K_M]\xrightarrow{\Psi} \Ind_{U_M}^M\psi_{\cZ}\xrightarrow{j^*} \Ind_{U_M}^M\psi_{\cZ_M}$$ differs from $\Lambda_{\cZ_M}^M$ by multiplication by a unit in $\cZ_M$.
\end{prop}

The strategy involves extending scalars to $\kappa := \Frac(\cZ)$, where it is easier to check the compatibility of the two maps because $\Frac(\cZ)$ is a field of characteristic zero. We will make repeated use of the fact that $\Frac(\cZ_M)=\cZ_M\otimes_{\cZ}\kappa$. In particular, if a $\cZ_M$-module is torsion-free when restricted to $\cZ$, it is also torsion-free as a $\cZ_M$-module.


For any smooth $W(k)[G]$-module $V$ recall that $*$ denotes the action of $\cZ$ on $V^K$ via double-coset convolution operators. The Bernstein presentation gives an inclusion of algebras $\cZ\cong W(k)[\Lambda]^{W_G}\to \cH(G,I)$ whose image is the center of $\cH(G,I)$. By virtue of this inclusion, $\cZ$ acts on any $\cH(G,I)$-module. Similarly, the embedding $j:\cZ\to \cZ_M$ gives an action of $\cZ$ on any $\cZ_M$-module, for example, $\cZ$ acts on $X^{K_M}$ for any $X\in \Rep_{W(k)}(M)$. Finally, $\cZ$ also acts on $\cH(M,I_M)$-modules, via the embeddings $\cZ\hookrightarrow \cZ_M \hookrightarrow \cH(M,I_M)$. Recall that the Bernstein presentation gives an embedding $\cZ_T\hookrightarrow \cH(G,I)$ whose image we denote $\cA$. We will denote all these convolution operator actions by $*$.

On the other hand, the Whittaker spaces $\Ind_U^G\psi_{\cZ}$, $\Ind_{U_M}^M\psi_{\cZ}$, $\Ind_{U_M}^M\psi_{\cZ_M}$, $\Ind_U^G\psi_{\kappa}$, $\Ind_{U_M}^M\psi_{\kappa}$, $\Ind_{U_M}^M\psi_{\kappa}\otimes_{\cZ}\cZ_M$ carry actions of $\cZ$, $\cZ$, and $\cZ_M$, $\kappa$, $\kappa$, and $\Frac(\cZ_M)$, respectively, by multiplication of the values of Whittaker functions. We will denote all these actions by $\cdot$. If no action is specified it is implicitly assumed to be the $\cdot$ action.

We will compare the actions $*$ and $\cdot$ by computing the subspace of $(\Ind_{U_M}^M\psi_{\kappa})^{K_M}$ where the two $\cZ_M$-actions coincide \emph{on the subring $\cZ$:}
$$(\Ind_{U_M}^M\psi_{\kappa})^{K_M, (\cZ,*)=(\cZ,\cdot)}:=\{ v \in (\Ind_{U_M}^M\psi_{\kappa})^{K_M}:\ z*v = z\cdot v\ ,\ z\in \cZ\}.$$ 

To save notation, abbreviate to $W^0$ the Whittaker function $W_{\text{id}}^0\in (\Ind_U^G\psi_{\cZ})^{K,\text{id}}$ given by Corollary~\ref{shintanispacefreerankone}. Let $W^0|_M$ be its restriction to $M$, considered as an element of $(\Ind_{U_M}^M\psi_{\cZ})^{K_M}$. Let $X:=\cZ_M* W^0|_M$ be the $(\cZ_M,*)$-submodule of $(\Ind_{U_M}^M\psi_{\cZ})^{K_M}$ generated by $W^0|_M$.
\begin{lemma}\label{cyclicmodule}
The image of $X\otimes_{\cZ}\kappa$ in the embedding 
\begin{align*}
(\Ind_{U_M}^M\psi_{\cZ})^{K_M}\otimes_{\cZ}\kappa &\hookrightarrow (\Ind_{U_M}^M\psi_{\kappa})^{K_M}\\
W\otimes \xi &\mapsto W\cdot \xi
\end{align*}
equals $(\Ind_{U_M}^M\psi_{\kappa})^{K_M, (\cZ,*)=(\cZ,\cdot)}$.
\end{lemma}
\begin{proof} 
Let $Y:= \cZ_T* W^0|_M$ be the $(\cZ_T,*)$-submodule of $(\Ind_{U_M}^M\psi_{\cZ})^{I_M}$ generated by $W^0|_M$. The Lemma follows from two claims.
\\

\noindent\textbf{Claim 1:} The image of $Y\otimes_{\cZ}\kappa$ in the embedding 
$$(\Ind_{U_M}^M\psi_{\cZ})^{I_M}\otimes_{\cZ}\kappa \to (\Ind_{U_M}^M\psi_{\kappa})^{I_M}$$ is the space $(\Ind_{U_M}^M\psi_{\kappa})^{I_M,(\cZ,*)=(\cZ,\cdot)}$.
\\

Note that $I_T = K_T = T\cap K$, so Claim 1 implies the Lemma in the case $M=T$. The isomorphism of Lemma~\ref{iwahoriisom} is not only an isomorphism of $\kappa$-vector spaces, it is equivariant with respect to the $*$ action of $\cZ_M$. Combined with Lemma~\ref{philemma1}, we have $$(\Ind_{U_M}^M\psi_{\kappa})^{I_M} \cong (\Ind_{\{1\}}^T\psi_{\kappa})^{K_T}$$ as $\kappa$-vector spaces and as $(\cZ_M,*)$-modules. As $\cZ$ acts via the natural embeddings $\cZ\hookrightarrow \cZ_M\hookrightarrow \cZ_T\hookrightarrow \cH(T,I_T)$, we have $$(\Ind_{U_M}^M\psi_{\kappa})^{I_M, (\cZ,*)=(\cZ,\cdot)} \cong (\Ind_{\{1\}}^T\psi_{\kappa})^{K_T, (\cZ,*)=(\cZ,\cdot)}.$$ An element of $(\Ind_{\{1\}}^T\psi_{\kappa})^{K_T}$ is a smooth function on the lattice $\varpi^{\ZZ^n}$. It is proved in \cite[Thm 3.2]{lazarus} that adding the additional condition that the actions $(\cZ,*)$ and $(\cZ,\cdot)$ be equivalent cuts out a space of dimension $n!$ over $\kappa$. So the space $(\Ind_{U_M}^M\psi_{\kappa})^{I_M, (\cZ,*)=(\cZ,\cdot)}$ has dimension $n!$ over $\kappa$.

On the other hand, by definition of $W^0$, the restriction $W^0|_M$ lies in the space $$(\Ind_{U_M}^M\psi_{\cZ})^{K_M,(\cZ,*)=(\cZ,\cdot)},$$ which is stable under the $*$ action of $\cZ_T$. Thus the cyclic module $Y$ is contained in $(\Ind_{U_M}^M\psi_{\cZ})^{I_M,(\cZ,*)=(\cZ,\cdot)}$. Since the action of $\cZ$ via $\cdot$ on $(\Ind_{U_M}^M\psi_{\cZ})^{I_M,(\cZ,*)=(\cZ,\cdot)}$ is torsion-free, so is the action of $\cZ$ via $*$ and hence so is the action of $\cZ_T$ via $*$.  Hence $Y$ is free of rank one over $\cZ_T\cong W(k)[\Lambda]\cong cA$, and $Y\otimes_{\cZ}\kappa$ has dimension $n!$ over $\kappa$. Since it is contained in $(\Ind_{U_M}^M\psi_{\kappa})^{I_M, (\cZ,*)=(\cZ,\cdot)}$, we have proven Claim 1.
\\

\noindent\textbf{Claim 2:} $(Y\otimes_{\cZ}\kappa)^{K_M} =X\otimes_{\cZ}\kappa$. 
\\

Note that $\ell$ is invertible in $\kappa$, hence the action of $\cH(M,I_M)$ factors through $\cH(M,I_M)[\frac{1}{\ell}] =\cH(M,I_M)\otimes_{W(k)}\cK$, where $\cK = \Frac(W(k))$. We will abbreviate
$$(-)_{\cK}:=(-)\otimes_{W(k)}\cK.$$ Let $Y'$ be the cyclic $\cH(M,I_M)$-submodule of $(\Ind_{U_M}^M\psi_Z)^{I_M}$ generated by $W^0|_M$. Then we can identify $Y'\otimes_{\cZ}\kappa$ with the $\cH(M,I_M)_{\cK}$-submodule of $(\Ind_{U_M}^M\kappa)^{I_M}$ generated by $W^0|_M$.

Since $W^0|_M$ is fixed by $K_M$, the subring $\cH_{W_M}=\cK[I_M\backslash K_M /K_M]$ of $\cH(M,I_M)_{\cK}$ acts by the ``trivial'' character $\textbf{1}:\cH_{W_M}\to \cK^{\times}$, which sends $[IwI]$ to $[IwI:I]$, which is a power of $q$, and invertible everywhere. Hence we have an equality of $\cH(M,I_M)_{\cK}$-modules:
$$Y'\otimes \kappa :=\cH(M,I_M)_{\cK}*W^0|_M =\cA_{\cK}* W^0|_M =(\cZ_T)_{\cK}*W^0|_M = Y\otimes \kappa.$$ Since $\cZ_T$ acts torsion-freely on $W^0|_M$, we have $Y\otimes\kappa \cong (\cZ_T)_{\cK},$ and Claim 2 follows since $(\cZ_T)_{\cK}^{K_M} = (\cZ_T)_{\cK}^{\cH_{W_M}} = \cK[\Lambda]^{W_M} = (\cZ_M)_{\cK}$.
\end{proof}

We conclude this section by proving Proposition~\ref{differbyunit}.

\begin{proof}[Proof of Proposition~\ref{differbyunit}]
Since $\Psi$ is $M$-equivariant, its restriction to $K_M$-fixed vectors is $(\cZ_M,*)$-equivariant, and hence also $(\cZ,*)$-equivariant. In particular $\Psi(1_{K_M})$ defines an element of the space
$$(\Ind_{U_M}^M\psi_{\cZ})^{K_M,(\cZ,*)=(\cZ,\cdot)}$$ defined before Lemma~\ref{cyclicmodule}. By the conclusion of Lemma~\ref{cyclicmodule}, there is an element $z$ of $\cZ_M$ and an element $z_0$ of $\cZ$ such that $z_0\Psi(1_{K_M})=z*W^0|_M$. If we consider $W^0|_M$ as an element of $\Ind_{U_M}^M\psi_{\cZ_M}$ by pushing it forward along the inclusion $j:\cZ\hookrightarrow \cZ_M$, we have $z_0j(\Psi(1_{K_M})) = z*W^0|_M$. By Proposition~\ref{philemma1}, $j(\Psi(1_{K_M})) = \Phi_{\cZ_M}'(W_{\text{id}}^0)$. It follows from Lemma~\ref{iwahoriisom} that $\Phi_{\cZ_M}'(W_{\text{id}}^0)$, and hence $z$, are nonzero.

On the other hand, the canonical element $$W^{0,M}_{\text{id}}\in(\Ind_{U_M}^M\psi_{\cZ_M})^{K_M,\text{id}}=:(\Ind_{U_M}^M\psi_{\cZ_M})^{K_M,(\cZ_M,*)=(\cZ_M,\cdot)}$$ also lies in the bigger space $(\Ind_{U_M}^M\psi_{\cZ_M})^{K_M,(\cZ,*)=(\cZ,\cdot)}.$ Extending scalars from $\kappa$ to $\Frac(\cZ_M)$ in the conclusion of Lemma~\ref{cyclicmodule}, we find there is some $z'\in \cZ_M$ and $z_1\in \cZ$ such that $z_1W^{0,M}_{\text{id}}=(z')*W^0|_M$. Since $W_{\text{id}}^{0,M}(1)=1$, $z'$ is nonzero.

Since $\Frac(\cZ_M) = \cZ_M\otimes_{\cZ}\Frac(\cZ)$, we can write $$\frac{z}{z'}= \frac{z''}{z_2}\in \Frac(\cZ_M)\text{, \ for some $z''\in \cZ_M$, $z_2\in \cZ$}.$$ We have 
$$z_0 j(\Psi(1_{K_M})) =z''* (\frac{z_1}{z_2}\cdot W^{0,M}_{\text{id}}) = \frac{z''z_1}{z_2}\cdot W^{0,M}_{\text{id}}\ \text{ in }(\Ind_{U_M}^M\psi_{\Frac(\cZ_M)})^{K_M}.$$ In other words,
$$z_0z_2j(\Psi(1_{K_M})) = z''z_1\cdot W^{0,M}_{\text{id}}\text{ in }(\Ind_{U_M}^M\psi_{\cZ_M})^{K_M}.$$ Thus the action of $\cZ_M$ on $z_0z_2j(\Psi(1_{K_M}))$ via $*$ is equivalent to that via $\cdot$, since it is true for any multiple of $W^0|_M$. It follows from $\cZ$-torsion freeness that the same is true for $j(\Psi(1_{K_M}))$ itself. Thus $j(\Psi(1_{K_M}))$ lies in the space
$$(\Ind_{U_M}^M\psi_{\cZ_M})^{K_M,(\cZ_M,*) = (\cZ_M,\cdot)} = (\Ind_{U_M}^M\psi_{\cZ_M})^{K_M,\text{id}}=\cZ_M\cdot W_{\text{id}}^{0,M}.$$ 

By Lemma~\ref{iwahoriisom}, we already know that $j(\Psi(1_{K_M})) = \Phi_{\cZ_M}'(W^0_{\text{id}})$ is nonzero modulo every maximal ideal of $\cZ_M$, hence the same must be true of $j(\Psi(1_{K_M}))(1)$ by Lemma~\ref{shintaniforlevi}. In particular $j(\Psi(1_{K_M}))(1)$ is a unit in $\cZ_M$. By Lemma~\ref{shintaniforlevi}, $j(\Psi(1_{K_M}))$ differs from $W^{0,M}_{\text{id}}$ by multiplication by $j(\Psi(1_{K_M}))(1)$. In other words, $j(\Psi(1_{K_M}))$ is equivalent to $\Lambda_{\cZ_M}^M(1_{K_M})$ after multiplying by a unit in $\cZ_M$. Since the maps $j\circ\Psi$ and $\Lambda_{\cZ_M}^M$ are $W(k)[M]$-equivariant, and $W(k)[M/K_M]$ is generated as a $W(k)[M]$-module by $1_{K_M}$, the statement of the Proposition follows.
\end{proof}

We emphasize that the map $$j \circ \Psi:W(k)[M/K_M]\to \Ind_{U_M}^M\psi_{\cZ_M}$$ was \emph{a priori} only $\cZ[M]$-equivariant, but we have now established that it is in fact $\cZ_M[M]$-equivariant, as it coincides with $\Lambda_{\cZ_M}^M$ up to a unit. We also emphasize that the scalar in $\cZ_M^{\times}$ by which $j\circ\Psi$ differs from $\Lambda_{\cZ_M}^M$ depends on the choice of isomorphism $r_P^G(\Ind_U^G\psi_{\cZ_M})\cong \Ind_{U_M}^M\psi_{\cZ_M}$ made in the definition of $\Phi'$ (hence of $\Psi$) in Lemma~\ref{philemma1}.

\section{The main theorem}\label{mainthm:section}

Recall that $\cZ:=W(k)[K\backslash G/K]$, $R$ is a ring, $\lambda$ is a homomorphism $\cZ\to R$, and $\cM_{\lambda,R}:=W(k)[G/K]\otimes_{\cZ,\lambda}R$.

Let $W_{\lambda}^0$ be the canonical element of $\Ind_U^G\psi_R$ guaranteed by Lemma~\ref{Shintani}, and let $\Lambda_R:\cM_{\lambda,R}\to \Ind_U^G\psi_R$ be the canonical map given by $1_K\otimes 1\mapsto W^0_{\lambda}$.

\begin{thm}\label{universalmodule:conj}
Let $R$ be a Noetherian ring. The map $\Lambda_R:\cM_{\lambda,R}\to \Ind_U^G\psi_R$ is injective.
\end{thm}

The proof of Theorem~\ref{universalmodule:conj} will occupy the remainder of this section. The strategy is to use induction, combined with Lemma~\ref{jacquetmoduleM}, Lemma~\ref{cuspidal=generic}, Proposition~\ref{philemma1}, and Proposition~\ref{differbyunit}.
\begin{proof}[Proof of Theorem~\ref{universalmodule:conj}]
Let $n=1$. Then $$\cZ = W(k)[F^{\times}/\cO_F^{\times}]\cong W(k)[\varpi^{\ZZ}]\cong W(k)[X_1^{\pm 1}].$$ The module $W(k)[G/K] = \cZ$ is free of rank 1 over $\cZ$. Therefore $\cM_{\lambda,R}\cong R$ with the action of $F^{\times}$ given by the unramified character $\lambda$. Since $U=\{1\}$, we have $\Ind_U^G\psi_R = C^{\infty}(F^{\times},R)$, and 
\begin{equation*}
W^0_{\lambda}(x)=
\left\{
\begin{array}{ll} \lambda(X_1^{v_F(x)})&\text{ if }v_F(x)\geq 0\\
0& \text{ otherwise}.
\end{array}\right.
\end{equation*}
In this case, the $F^{\times}$-equivariant map $\Lambda_R:R \to C^{\infty}(F^{\times},R)$ given by $1\mapsto W^0_{\lambda}$ is certainly injective since $W^0_{\lambda}(1)=1$.

Now let $n$ be arbitrary. Choose a proper parabolic subgroup with Levi decomposition $P=MN$ and consider the map $\Psi$ defined by Eqn~(\ref{maindiagram}). By Proposition~\ref{differbyunit}, the following diagram commutes, up to a scalar multiple in $\cZ_M^{\times}$:
 \begin{equation*}
\xymatrix{
W(k)[G/K] \ar[r]^{\Lambda_{\cZ}}\ar[d]^{\Phi} &\Ind_U^G\psi_{\cZ} \ar[r] &\Ind_U^G\psi_{\cZ_M}\ar[d]^{\Phi'_{\cZ_M}}\\
W(k)[M/K_M] \ar[rr]^{\Lambda^M_{\cZ_M}} && \Ind_{U_M}^M\psi_{\cZ_M}.
}
\end{equation*}

Let $R'$ be the $\cZ_M$-algebra $R\otimes_{\cZ}\cZ_M$. We make the identification of $\cZ$-modules
$$W(k)[M/K_M]\otimes_{\cZ}R \cong (W(k)[M/K_M]\otimes_{\cZ_M}\cZ_M)\otimes_{\cZ}R\cong W(k)[M/K_M]\otimes_{\cZ_M}R'.$$ By tensoring with $R$ and composing with the natural maps 
\begin{align*}
(\Ind_U^G\psi_{\cZ})\otimes_{\cZ}R&\to \Ind_U^G\psi_R\\
W\otimes \xi &\mapsto \lambda^*W\cdot \xi,
\end{align*} we obtain the following commutative-up-to-unit-scalar diagram of $\cZ$-modules:
\begin{equation*}
\xymatrix{
W(k)[G/K]\otimes_{\cZ}R \ar[r]^{\Lambda_{R}}\ar[d]^{\Phi\otimes\text{id}} &\Ind_U^G\psi_{R} \ar[r] &\Ind_U^G\psi_{R'}\ar[d]^{\Phi'_{R'}}\\
W(k)[M/K_M]\otimes_{\cZ_M}R' \ar[rr]^{\Lambda^M_{R'}} && \Ind_{U_M}^M\psi_{R'}.
}
\end{equation*}

Because the diagram is commutative up to a unit scalar in $R'$, we have $$(\Lambda_{R'}^M\circ(\Phi\otimes\text{id}))(\ker\Lambda_R)=0.$$ But $M$ is a proper Levi subgroup, so it is a product of groups $GL_{n_i}(F)$ along the diagonal, for $n_i<n$. The ring $R'$ is Noetherian because $R$ is. Thus we can apply the induction hypothesis to $$W(k)[M/K_M]\otimes_{\cZ_M}R'$$ to conclude $\Lambda^{M}_{R'}$ is injective. Therefore $(\Phi\otimes\text{id})(\ker\Lambda_R)=0$. Since $\Phi\otimes\text{id}$ factors as the composition $$\cM_{\lambda,R}\xrightarrow{p_N} r_P^G(\cM_{\lambda,R})\cong W(k)[M/K_M]\otimes_{\cZ_M}R',$$ we conclude that $p_N(\ker\Lambda_R)=0$, and hence $r_P^G(\ker\Lambda_R)=0$ for all proper parabolic subgroups $P$. By Proposition~\ref{admissibility}, $\cM_{\lambda,R}$ is admissible over $R$. Thus since $R$ is Noetherian Lemma~\ref{cuspidal=generic} tells us that either $\ker\Lambda_R=0$ or $(\ker\Lambda_R)^{(n)}\neq 0$.

We conclude the proof of Theorem~\ref{universalmodule:conj} by showing that $(\ker\Lambda_R)^{(n)}=0$. Indeed, the sequence
$$0\to \ker\Lambda_R\to \cM_{\lambda,R}\to \Ind_U^G\psi_R$$ is exact, hence so is
$$0\to (\ker\Lambda_R)^{(n)} \to (\cM_{\lambda,R})^{(n)} \xrightarrow{\Lambda_R^{(n)}} (\Ind_U^G\psi_R)^{(n)}.$$ If we let $\text{ev}_1:\Ind_U^G\psi_R\to R$ denote the map $W\mapsto W(1)$, then $\text{ev}_1\circ \Lambda_R$ factors through the composition $$\cM_{\lambda,R}\to (\cM_{\lambda,R})^{(n)}\xrightarrow{\Lambda_R^{(n)}} (\Ind_U^G\psi_R)^{(n)}\xrightarrow{\text{ev}_1^{(n)}} R.$$ But $\text{ev}_1\circ \Lambda_R$ is precisely the map inducing the isomorphism $$(\text{ev}_1\circ\Lambda_R)^{(n)}:(\cM_{\lambda,R})^{(n)}\isomto R$$ of Proposition~\ref{derivativefreerankone:cor}. In particular, $\Lambda_R^{(n)}$ is injective, and $(\ker\Lambda_R)^{(n)}$ is zero.
\end{proof}

\section{The global setup of the Ihara conjecture}\label{globalsetup:section}
\subsection{Unitary groups}
Let $B$ be a division algebra, equipped with an involution $*:B\to B$. Let $F$ denote the center of $B$, and let $F^+$ denote the subfield of $F$ fixed by $*$. The unitary group $U(B,*)_{/F^+}$ is the algebraic group over $F^+$ representing the functor:
$$R\mapsto \{g\in (B\otimes_{F^+}R)^{\times}: g^{*\otimes 1}g=1\}\text{,\ \ \  for $F^+$-algebras $R$}.$$ 

The involution $*$ is said to be \emph{of the second kind} if $F\neq F^+$. Then $[F:F^+]=2$ and $B\otimes_{F^+} \overline{F^+}=(B\otimes_{F^+}F)\otimes_{F}\overline{F^+}\cong M_n(\overline{F^+})\times M_n(\overline{F^+})$, where $n^2=\dim_F(B)$. After possibly modifying this isomorphism by an inner automorphism, $*$ is given on the target by $(X,Y)\mapsto (^tY, {^tX})$ (\cite[2.3.3]{pr}). Thus, we have
\begin{align*}
U(B,*)(\overline{F^+})&\cong \{(g,h)\in GL_n(\overline{F^+})\times GL_n(\overline{F^+}): (g,h)({^th}, {^tg})=(I,I)\}\\
&=\{(g,{^tg^{-1}})\in GL_n(\overline{F^+})\times GL_n(\overline{F^+})\}\\
&\cong GL_n(\overline{F^+})\text{\ \ (non-canonically),}
\end{align*}
so $U(B,*)$ is an inner form of $GL_n$.

Now let $F^+$ be a fixed totally real number field, and $F=EF^+$ an imaginary quadratic extension that is unramified everywhere, and let $c$ denote conjugation in $\Gal(F/F^+)$. We will construct a division algebra $B$ with center $F$, equipped with an involution $*$ extending $c$, of the second kind, such that $U(B,*)$ has certain properties desirable for studying automorphic forms.
\\

\noindent\textbf{Note:} The letter $F$ henceforth denotes a number field, whereas in all previous sections it denoted a $p$-adic field. We will eventually apply the results of the previous sections to the completions $F_w$ of $F$.
\\

We follow the notation and choices of \cite[\S 3.3-3.5]{cht}.  Let $\ell>n>1$ be a prime that splits in $F/F^+$ and let $S_{\ell}$ be the set of places of $F^+$ above $\ell$.

Choose a nonempty finite set of places $S(B)$, each of which splits in $F/F^+$, none of which lies in $S_{\ell}$, and such that, when $n$ is even, $\# S(B)$ has the same parity as $\frac{n}{2}[F^+:\QQ]$. Then there exists a division algebra $B$ with center $F$, of dimension $n^2$ over $F$, which is non-split at the places lying over $S(B)$, and such that $B^{op}\cong B\otimes_{E,c}E$. According to \cite[\S3.3]{cht}, the assumption that $\# S(B)$ has the same parity as $\frac{n}{2}[F^+:\QQ]$ for even $n$ makes it possible to extend $c$ to an involution $*:B\to B$ of the second kind, in such a way that the group $G = U(B,*)$ satisfies: \begin{itemize}
\item the unitary group $G$ is \emph{compact at infinity}, i.e., $G(F_v^+)\cong U(n)$ for all $v\in S_{\infty}$, 
\item at all the finite places $v\notin S(B)$ that do not split in the quadratic extension $F/F^+$, the unitary group is \emph{quasi-split}, i.e., $G(F_v^+)$ is a quasi-split unitary group over $F_v^+$.
\end{itemize}

For each place $v$ of $F^+$ that splits in $F$, the two places $w, \bar{w}$ lying over $F$ are exchanged by $c$, and the completions $F_w=F_{\bar{w}}$ are both equal to $F^+_v$. Choosing a place $w$ lying over $v\not\in S(B)$ is equivalent to choosing an isomorphism
\begin{align}\label{glnsplitplaces}
G(F_v^+) &= \{(g,h)\in GL_n(F_w)\times GL_n(F_{\bar{w}}): (g,h)(^th,^tg)=(I,I)\}\\
&= \{(g, {^tg^{-1}}) \in GL_n(F_w)\times GL_n(F_{\bar{w}})\}\nonumber\\
&\cong GL_n(F_w),\nonumber
\end{align}
and the two isomorphisms differ by $^t(-)^{-1}$. Moreover, it is possible to choose an order $\cO_B$ in $B$ that is stable under $*$ and such that $\cO_{B,w}$ is maximal for all places $w$ of $F$ lying over split places of $F^+$; this gives a model of $G$ defined over $\cO_{F^+}$, such that $G(\cO_{F^+,v})$ is isomorphic under (\ref{glnsplitplaces}) to $GL_n(\cO_{F,w})$ at split places $v\not\in S(B)$.

\subsection{Automorphic forms on definite unitary groups}
Let $H$ be a connected reductive group over the totally real field $F^+$. Let $\bA_{F^+}$ denote the ring of $F^+$-adeles and $\bA_{F^+}^{\infty}$ the ring of finite adeles. Then $H(\mathbb{A}_{F^+})$ decomposes as a restricted direct product $\prod_v' H(F^+_v)$, over all places $v$ of $F^+$, where the components lie in $H(\cO_{F^+,v})$ for all but finitely many $v$. Given a subset $S$ of places, a superscript $X^S$ will always denote $\prod_{v\not\in S}X_v$ and and a subscript $X_S$ will denote $\prod_{v\in S}X_v$. 

If $K_{\infty}$ is a maximal compact subgroup of the infinite adeles $H(F^+_{\infty})$, and $U$ is a compact open subgroup of the finite adeles $H(\bA_{F^+}^{\infty})$, a classical (weight 0) automorphic form of level $K_{\infty}U$ on $H$ is a $\CC$-valued function on the double coset space $$H(F^+)\backslash H(\bA_{F^+})/K_{\infty}U,$$ that satisfies various smoothness and growth conditions. $H$ is called \emph{definite} if $H(F^+_{\infty})$ is itself compact, in which case $K_{\infty} = H(F^+_{\infty})$, and it follows that
$$H(F^+)\backslash H(\bA_{F^+})/K_{\infty}U = H(F^+)\backslash H(\bA_{F^+}^{\infty})/U,$$ which by a theorem of Borel and Harish-Chandra is \emph{finite}. It was observed by Gross in \cite{gross} that automorphic forms on definite groups provide a good framework for investigating algebraic properties of automorphic forms, such as congruences.

Note that our particular group $G=U(B,*)$ is a definite unitary group because $G(F_v^+)\cong U(n)$ for all $v\in S_{\infty}$. Beyond its being definite, Clozel, Harris, and Taylor consider this particular species of unitary group in \cite{cht} because its automorphic forms are well-studied, especially with respect to base-change to $GL_n$, and its algebraic automorphic forms give rise to Galois representations valued in a group $\mathcal{G}$ closely connected to $GL_n$.

Algebraic automorphic forms on $G$ of level $U$ (and weight 0) are functions on $$G(F^+)\backslash G(\bA_{F^+}^{\infty})/U$$ valued in rings more general than $\CC$. \emph{In this section only,} let $k$ be an arbitrary algebraic extension of $\FF_{\ell}$, let $W(k)$ denote the Witt vectors and let $\cK = \Frac(W(k))$ or a ``sufficiently large'' finite extension of $\Frac(W(k))$. Let $\cO$ denote the ring of integers in $\cK$. Fix an isomorphism $\iota:\overline{\cK}\isomto \CC$.

We will specify a level, i.e., a compact open subgroup $U=\prod_vU_v \subset G(\bA_{F^+}^{\infty})$ by fixing various sets of places and requiring that $U_v$ satisfy certain conditions for $v$ in those sets. We would like $U$ to be \emph{sufficiently small}, which means some $U_v$ contains no non-trivial elements of finite order. Fix a finite nonempty set $S_a$ of finite places, each of which is split in $F/F^+$, such that $S_a$ is disjoint from $S_{\ell}\cup S(B)$ and such that, if $v|p$, then $[F(\zeta_p):F]>n$. We assume that $U_v\cong I + \varpi_{v}M_n(\cO_{F^+_v})$ for $v\in S_a$, and this guarantees that $U$ is sufficiently small.

For any $\cO$-algebra $A$, let $S(U,A)$ be the set of locally constant functions
$$\{f: G(F^+)\backslash G(\bA_{F^+}^{\infty}) \to A, \  f(gu) = f(g), u\in U, g\in G(\bA_{F^+}^{\infty})\}$$ for all $u\in U$. Since $U$ is sufficiently small this is a finite free $A$-module (\cite[p.98]{cht}). It is a space of $\ell$-integral automorphic forms in the sense that
\begin{align*}
S(U,\cO)\otimes_{\cO}\overline{\cK} &\isomto \bigoplus_{\pi}\left(\bigotimes_{v\text{ finite}}\pi_v^{U_v}\right)^{\oplus m_{\pi}}\\
f&\mapsto [\phi: g\mapsto \iota(f(g^{\infty}))]
\end{align*}
where $\pi$ runs over all classical automorphic representations of $G(\bA_{F^+})$ over $\mathbb{C}$ such that $\pi_{\infty}$ is the trivial representation, and $m_{\pi}$ is the multiplicity of $\pi$.

From \cite[Prop 9.2]{gross}, this construction is compatible with reduction mod-$\ell$: 
\begin{align*}
S(U,\cO)\otimes_{\cO}k &= S(U,k)\\ 
\End_{\cO}(S(U,\cO))\otimes k &\cong\End_{k}(S(U,k)).
\end{align*}

\subsection{Global Hecke algebras}
Let $T$ be a finite set of finite places containing $S_{\ell}\cup S_B \cup S_a$, all of which split in $F/F^+$. Assume $U_v\cong G(\cO_{F_v^+})\cong GL_n(\cO_{F_v^+})$ for all split places $v\not\in T$, and $U_v$ is hyperspecial for all non-split places $v\not\in T$. Each of the two divisors $w|v$, for $v\not\in T$, gives an isomorphism $$G(F^+_v)\cong GL_n(F_w)=GL_n(F_v^+).$$ Define the Hecke operators $T_w^{(j)}$, $j=1,\dots,n$ to be the double coset operators
$$T_w^{(j)} = [GL_n(\cO_{F_w})\left(\begin{smallmatrix}\varpi_wI_j&\\&I_{n-j} \end{smallmatrix}\right) GL_n(\cO_{F_w})].$$
Define
$$\TT^T(U) := \cO\left[T_w^{(1)},\dots,T_w^{(n)}, (T_w^{(n)})^{-1}:w|v\notin T,\ v\text{ split}\right]$$ to be the $\cO$-subalgebra of $\End_{\cO}(S(U,\cO))$ generated by the operators $$T_w^{(1)},\dots,T_w^{(n)}, (T_w^{(n)})^{-1},$$ for any choice of $w|v$ where $v$ ranges over split places not in $T$. Then $\TT^T(U)$ is a reduced commutative ring which is finite free as a $\cO$-module (\cite[\S 3.4]{cht}), and does not depend on the choices of $w|v$.

Let $\Gamma_F$ be the absolute Galois group. To each maximal ideal $\fm\subset \TT^T(U)$ there is associated a unique continuous semisimple Galois representation
$$\overline{r}_{\fm}:\Gamma_F \to GL_n(\TT^T(U)/\fm)=GL_n(k),$$ unramified away from $T$, such that the $j$'th coefficient of the characteristic polynomial of Frobenius away from $T$ is $(-1)^j(\#k(v))^{j(j-1)/2}T_w^{(j)}$ mod $\fm$, where $\#k(v)$ denotes the order of the residue field of $F^+_v$ (\cite[Prop 3.4.2]{cht}). 

If $\overline{r}_{\fm}$ is absolutely irreducible, it has a natural continuous lifting $r_{\fm}$ to the localization $\TT^T(U)_{\fm}$. We say $\fm$ is \emph{non-Eisenstein} if $\overline{r}_{\fm}$ is absolutely irreducible.

\begin{prop}[\cite{cht} Cor 3.4.5]
\label{sphericalheckeeigenlift}
Suppose $\fm$ is non-Eisenstein and $v_0\in T- (S_{\ell}\cup S(B))$ and $U_{v_0}=G(\cO_{F^+,v_0})$. If $w$ is a prime of $F$ above $v_0$ then there exist $t_1,\dots t_n\in \TT^T(U)_{\fm}$ such that 
$$T_w^{(j)}*f = t_jf,\ \text{ for j=1,\dots,n}$$ for any $f$ in $S(U,\cO)_{\fm}$.
\end{prop}

\section{Application of Theorem~\ref{universalmodule:conj} to Ihara}\label{applicationtoihara:section}

The $R=\TT$ theorem of \cite{cht} is proven conditionally on a conjecture, known as Ihara's lemma. As explained in the introduction, we can apply Theorem~\ref{universalmodule:conj} to reduce this conjecture to an easier statement. We now give more details.

From this section onward, we reinstate our assumption that $k$ is algebraically closed.

\begin{conj}[\cite{cht} \S5.3: weak Ihara's lemma]\label{weakihara:conj}
Let $U\subset G(\bA_{F^+}^{\infty})$ be a sufficiently small open subgroup. Suppose
\begin{itemize}
\item $v_0\in T - (S_{\ell}\cup S(B)\cup S_a)$ is a place where $U_{v_0}\cong G(\cO_{F_{v_0}^+})$,
\item $\fm\subset \TT^T(U)$ is a non-Eisenstein maximal ideal,
\item $f$ is an element of $S(U,k)[\fm]$.
\end{itemize}
If $\pi$ is an irreducible $k[G(F_{v_0}^+)]$-submodule of $$\langle G(F_{v_0}^+)\cdot f\rangle\subset S(U^{v_0},k),$$ then $\pi$ is generic.
\end{conj}

Toward the goal of Conjecture~\ref{weakihara:conj}, we have the following.
\begin{thm}\label{localprops}
For $U$, $v_0$, and $\fm$ as in Conjecture~\ref{weakihara:conj}, $S(U,k)[\fm]$ is a spherical Hecke eigenspace at $v_0$. More precisely if $w|v_0$, there is a homomorphism $$\lambda:k[GL_n(\cO_{F,w_0})\backslash GL_n(F_{w_0})/ GL_n(\cO_{F,w_0})]\to k$$ such that 
$$T_{w}^{(j)}* f = \lambda(T_{w}^{(j)})f\ \text{\ \  $j=1,\dots,n$},$$ for all $f\in S(U,k)[\fm]$.
\end{thm}

\begin{proof}
This is just a formal consequence of Proposition~\ref{sphericalheckeeigenlift}. Take $U$, $v_0$, and $\fm$ as in Conjecture~\ref{weakihara:conj}.

Since $\TT^T(U)$ is finite free over the complete DVR $\cO$, it is semilocal, and we can write $\TT^T(U) = \prod_{\fm'}\TT^T(U)_{\fm'}$, a product over the localizations at each maximal ideal. Then $S(U,\cO)$ and $S(U,k)$ decompose as the product of their localizations $\prod_{\fm'}S(U,\cO)_{\fm'}$ and $\prod_{\fm'}S(U,k)_{\fm'}$, respectively. In particular, $S(U,k)[\fm]=S(U,k)_{\fm}[\fm]$.

We also have $S(U,\cO)_{\fm}\otimes_{\cO}k=S(U,k)_{\fm}$, and there is a natural map 
$$\End_{\cO}(S(U,\cO)_{\fm})\to \End_k(S(U,k)_{\fm}).$$ The image of each Hecke operator $T_w^{(j)}$ in this map is the Hecke operator $T_w^{(j)}$, by definition.

Thus we conclude that the action of the Hecke operator $T_w^{(j)}$ on $S(U,k)_{\fm}$ is given by the reduction mod-$\ell$ of the scalar $t_j\in \TT^T(U)_{\fm}$ appearing in Proposition~\ref{sphericalheckeeigenlift}. 

The action of $\TT^T(U)_{\fm}\otimes_{\cO}k$ on $S(U,k)_{\fm}[\fm]$ factors through the residue field $\TT^T(U)_{\fm}/\fm=k$. We conclude that there are scalars $\overline{t_j}\in k$ such that the action of $T_w^{(j)}$ on $S(U,k)_{\fm}[\fm]$ is given by multiplication by $\overline{t_j}$. The result follows.
\end{proof}

We now state a weaker conjecture:
\begin{conj}\label{weakerihara:conj}
Assume the setup of Conjecture~\ref{weakihara:conj}. The $k[G(F_{v_0}^+)]$-module $\langle G(F_{v_0}^+)\cdot f\rangle$ is generic.
\end{conj}

As a corollary of Theorem~\ref{universalmodule:conj}, we obtain the following.

\begin{cor}
Conjecture~\ref{weakerihara:conj} and Conjecture~\ref{weakihara:conj} are equivalent, and both are equivalent to $\langle G(F_{v_0}^+)\cdot f\rangle$ having exactly one irreducible Jordan--Holder constituent that is generic.
\end{cor}
\begin{proof}
Theorem~\ref{localprops} shows that the $k[G(F_{v_0}^+)]$-module $\langle G(F_{v_0}^+)\cdot f\rangle$ satisfies hypotheses (1) and (2) of Corollary~\ref{propertiesneeded}. Thus the result is given by the conclusion of Corollary~\ref{propertiesneeded}, combined with the conclusion of Theorem~\ref{universalmodule:conj}.
\end{proof}

For its applications to $R=\TT$ theorems, it suffices to know the truth of Conjecture~\ref{weakihara:conj} in the quasi-banal setting (c.f. \cite[Prop 5.3.5]{cht}), where we can refine our result slightly thanks to our work in Subsection~\ref{quasibanal:section}. In the notation of Conjecture~\ref{weakihara:conj}, let $I_{v_0}$ be the Iwahori subgroup of $G(F_{v_0}^+)\cong GL_n(F_{w_0})$, let $\cH(G(F_{v_0}^+),I_{v_0})$ be the local Iwahori--Hecke algebra at $v_0$, and let $\cA_{v_0}$ be the subalgebra $k[I_{v_0}\backslash G(F_{v_0}^+)/ G(\cO_{F^+,v_0})]$.

\begin{cor}\label{quasibanaliwahorirestatement}
Let $v_0$, $f$, $U$ be as in Conjecture~\ref{weakihara:conj}, suppose $\ell$ is quasi-banal for $\#k(v_0)$, and let $\widetilde{U} := (\prod_{v\neq v_0}U_v)\times I_{v_0}$. Then Conjecture~\ref{weakihara:conj} is equivalent to the following statement: the cyclic $\cA_{v_0}$-submodule of $S(\widetilde{U},k)$ generated by $f$ has dimension $n!$.
\end{cor}
\begin{proof}
By Theorem~\ref{localprops} there is some $\lambda:\cZ\to k$ and a map $\cM_{\lambda,k}\to \langle G(F_{v_0}^+)\cdot f\rangle$ whose image contains $f$. Thus $\langle G(F_{v_0}^+)\cdot f\rangle$ is a quotient of $\cM_{\lambda,k}$, hence $\langle G(F_{v_0}^+)\cdot f\rangle^I$ is a quotient of $\cM_{\lambda,k}^I$. Thus $\langle G(F_{v_0}^+)\cdot f\rangle^I$ is cyclic as an $\cH(G,I)$-module. Since $\cH_W$ acts trivially on $f$, we have that $\langle G(F_{v_0}^+)\cdot f\rangle^I\cong \cA* f$ inside $S(\widetilde{U},k)$. Thus, if $\cA* f$ had dimension $n!$, $\langle G(F_{v_0}^+)\cdot f\rangle^I$ would be isomorphic to $\cM_{\lambda,k}^I$ and $\langle G(F_{v_0}^+)\cdot f\rangle$ would be isomorphic to $\cM_{\lambda,k}$. The other direction is immediate, since we have already established that Conjecture~\ref{weakihara:conj} is equivalent to the map $\cM_{\lambda,k}\to \langle G(F_{v_0}^+)\cdot f\rangle$ being an isomorphism, and $\cM_{\lambda,k}^I$ has dimension $n!$.
\end{proof}

\end{document}